\documentclass[12pt,a4paper]{article}
\usepackage[utf8]{inputenc}
\usepackage[T1]{fontenc}

\usepackage{vmargin}
\usepackage{amsmath}
\usepackage{amsthm}
\usepackage{amssymb}
\usepackage{amsopn}
\usepackage{enumitem}
\usepackage{color}
\usepackage{esint}
\usepackage{mathtools}
\usepackage{hyperref}
\usepackage{doi}

\DeclareMathOperator{\divo}{div}
\DeclareMathOperator{\Id}{Id}

\DeclareMathOperator{\dist}{dist}
\DeclareMathOperator{\diam}{diam}
\DeclareMathOperator{\co}{co}
\DeclareMathOperator{\spano}{span}

\DeclareMathOperator{\Reo}{Re}
\DeclareMathOperator{\Imo}{Im}

\DeclareMathOperator{\sinto}{sint}
\DeclareMathOperator{\into}{int}

\DeclareMathOperator{\Bil}{Bil}
\DeclareMathOperator*{\esssup}{ess\,sup}

\theoremstyle{plain}
\newtheorem{thm}{Theorem}[section]
\newtheorem{prop}[thm]{Proposition}
\newtheorem{cor}[thm]{Corollary}
\newtheorem{lem}[thm]{Lemma} 

\theoremstyle{definition}

\theoremstyle{remark}
\newtheorem{rem}[thm]{Remark}

\newcommand{\R}{\mathbb{R}}

\newcommand{\Z}{\mathbb{Z}}
\newcommand{\C}{\mathbb{C}}
\newcommand{\N}{\mathbb{N}}

\newcommand{\eps}{\varepsilon}

\begin{document}
\begin{center}
\begin{Large}
Existence and Convergence of Solutions of the \newline
Boundary Value Problem in Atomistic and \newline
Continuum Nonlinear Elasticity Theory \newline
\end{Large}
\\[0.25cm]
\begin{large}
Julian Braun\footnote{Universität Augsburg, Germany, {\tt julian.braun@math.uni-augsburg.de}} and Bernd Schmidt\footnote{Universität Augsburg, Germany, {\tt bernd.schmidt@math.uni-augsburg.de}}
\end{large}
\\[0.5cm]
\today
\\[1cm]
\end{center}

\begin{abstract}
We show existence of solutions for the equations of static atomistic nonlinear elasticity theory on a bounded domain with prescribed boundary values. We also show their convergence to the solutions of continuum nonlinear elasticity theory, with energy density given by the Cauchy-Born rule, as the interatomic distances tend to zero. These results hold for small data close to a stable lattice for general finite range interaction potentials. We also discuss the notion of stability in detail. 
\end{abstract}

\section{Introduction}

In classical continuum mechanics, the (static) behavior of an elastic body subject to applied body forces and imposed boundary values is described in terms of a deformation mapping which satisfies the (static) partial differential equations of nonlinear elasticity theory. For hyperelastic materials, these equilibrium equations of elastostatics are the Euler-Lagrange equations of an associated energy functional in which the stored elastic energy enters in terms of an integral of a stored energy function acting on the local deformation gradient. Stable configurations are given by deformations which are local minimizers of this functional. The stored energy density in particular induces the stress strain relation and thus encodes the elastic properties of such a material.

On a more basic level, crystalline solids may be viewed as particle systems consisting of interacting atoms on (a portion of) a Bravais lattice. The interatomic cohesive and repulsive forces, which are dominantly induced by the atomic electronic structure, can effectively be modeled in terms of classical interaction potentials. Stable configurations are now local minimizers of the total atomistic energy and solutions of the high dimensional system of equations for the force balance. 

The classical connection between atomistic and continuum models of nonlinear elasticity is provided by the Cauchy-Born rule: The continuum stored energy function associated to a macroscopic affine map is given by the energy (per unit volume) of a crystal which is homogeneously deformed with the same affine mapping. In particular, this entails the assumption that there are no fine scale oscillations on the atomistic scale. We will call this the Cauchy-Born energy density in the following. Indeed, if one assumes the Cauchy-Born rule to hold true and consequently requires that every individual atom follow a smooth macroscopic deformation mapping, one can derive a continuum energy expression for such a deformation from given atomistic potentials as shown by Blanc, Le Bris and Lions, cf.\ \cite{BLL:02}. To leading order in the small lattice spacing parameter $\eps$ one then obtains the continuum energy functional with the Cauchy-Born energy density. However, it is not clear a priori that the Cauchy-Born hypothesis is true. Moreover, it is desirable to not only obtain a pointwise convergence result for the corresponding energy functionals, but also to relate the solutions of the continuum problem to equilibrium configurations of the associated atomistic system. 

Our aim in this work is to establish a rigorous link between atomistic models in their asymptotic regime $\eps \to 0$ and the corresponding Cauchy-Born models from continuum mechanics for the nonlinear elastic behavior of crystalline solids accounting for body forces and boundary values. To be more precise, from a macroscopic point of view our goal is to show that, under suitable stability assumptions which effectively rule out atomistic relaxation effects, for each solution of the continuous boundary value problem there are solutions of the associated atomistic boundary value problems with lattice spacing $\eps$ which converge to the given continuum solution as $\eps \to 0$. Conversely, from the microscopic point of view, we aim at establishing sufficient conditions on atomistic body forces and boundary values that yield stable discrete solutions close to the corresponding continuum solution and thus obeying the Cauchy-Born rule. 

Over the last 15 years in particular there has been considerable progress in identifying conditions which allow for a mathematically rigorous justification of the Cauchy-Born rule. Here, we restrict our attention to those contributions which directly have influenced our results. For a general review on the Cauchy-Born rule we refer to the survey article \cite{Ericksen08} by Erickson. In their seminal contribution \cite{friesecketheil} Friesecke and Theil consider a two-dimensional mass-spring model and prove the Cauchy-Born rule does indeed hold true for small strains, while it in general fails for large strains. Their result has then been generalized to a wider class of discrete models and arbitrary dimensions by Conti, Dolzmann, Kirchheim and Müller in \cite{CDKM}. More specifically, in these papers a version of the Cauchy-Born rule is established by considering a box containing a portion of a crystal and showing that, under the condition that the atoms in a boundary layer (whose width depends on the maximal interatomic interaction length) follow a given affine deformation, the global minimizer of the energy is given by the homogeneous deformation in which all atoms follow that affine deformation. In \cite{braunschmidt13} we showed that these results can be combined with abstract results on integral representation to give a link in terms of $\Gamma$-convergence and, in particular, convergence of global minimizers of the atomistic energy to the continuum energy with Cauchy-Born energy density (for small strains) as the interatomic distances tend to zero. A corresponding discrete-to-continuum convergence result in which simultaneously the strain becomes infinitesimally small had been obtained by the second author in \cite{schmidtlinelast} resulting in a continuum energy functional with the linearized Cauchy-Born energy density. 

A drawback of these approaches which rely on global energy minimization is that they require strong growth assumptions on the atomic interactions that are not compatible with classical potentials such as, e.g., the Lennard-Jones potential. Based on the observation that elastic deformations in general are merely local energy minimizers, E and Ming have pioneered a different approach. In \cite{emingstatic} they show that, under suitable stability assumptions, solutions of the equations of continuum elasticity on the flat torus with smooth body forces are asymptotically approximated by corresponding atomistic equilibrium configurations. Recently these results have been generalized to a large class of interatomic potentials under remarkably mild regularity assumptions on the body forces for problems on the whole space by Ortner and Theil, cf.\ \cite{ortnertheil13}.  

In view of these results, the natural question arises if an analogous analysis is possible for a material occupying a general finite domain in space on the boundary of which there might also be prescribed boundary values. To cite Ericksen \cite[p.~207]{Ericksen08}, ``Cannot someone do something like this for a more realistic case, say zero surface tractions on part of the boundary and given displacements on the remainder?'' Our main result in Theorem \ref{thm:atomisticstatictheorem} will give an answer to this question. While we formulate our results only in the case of given displacements on the entire boundary, traction and mixed boundary conditions are automatically included. If one has a solution to the atomistic equations under given Dirichlet boundary conditions on a set of boundary atoms, one can just as easily declare these boundary atoms as non-physical ghost atoms that generate certain forces in their interaction range. Thus we also have a solution under a (specific) traction boundary condition. Our main restriction either way is that we can only consider a certain range of atomistic boundary data. But this is unsurprising. Indeed, as we will argue below, in the case of general atomistic boundary data the Cauchy-Born rule is typically expected to fail due to relaxation effects at the boundary.

We believe that our treatment of arbitrary domains and general displacement boundary conditions is of interest not only from a theoretical perspective but also with a view to specific situations that are of interest in applications, whose investigation indispensably requires an effective continuum theory. 

In order to relate our set-up to the aforementioned previous contributions, we remark that the presence of displacement boundary conditions leads to some subtleties within the statement of our main Theorem \ref{thm:atomisticstatictheorem} and to a number of technical difficulties within its proof: 1.\ As discussed above, boundary values are naturally imposed on a boundary layer of the atomistic system. In contrast to the situation, e.g., in \cite{CDKM}, the adequate choice of the atomistic displacements at the boundary for a general non-affine continuum boundary datum is not determined canonically a priori. In view of our rather mild regularity assumptions, one needs to construct the correct atomistic displacements from the continuum Cauchy-Born solution. Doing so we see that not only the correct asymptotic continuous boundary values but also the correct asymptotic normal derivative given by the normal derivative of the continuous Cauchy-Born solution are attained. (If these conditions fail, we again expect surface relaxation effects and a failure of the Cauchy-Born rule close to the boundary.) 2.\ In order to allow for as many atomistic boundary conditions (and body forces) as possible, we consider general scalings $\eps^\gamma$ in Theorem \ref{thm:atomisticstatictheorem} and only restrict $\gamma$ as much as necessary ($\frac{d}{2} \leq \gamma \leq 2$). While smaller $\gamma$ will lead to a larger variety of atomistic boundary values, $\gamma = 2$ will lead to optimal convergence rates. One should also note that our result no longer requires $\eps$ to be small. 3.\ Certain technical methods, which are available on the flat torus or on the whole space, do not translate to our setting. E.g., quasi-interpolations as in \cite{OrtnerShapeev12} do not preserve boundary conditions. This leads us to prove the important residual estimates, which lie at the core of our main proof, in a more direct way. With the help of a subtle atomic scale regularization, this can be achieved by requiring only slightly higher regularity assumptions for the continuous equations as compared to \cite{ortnertheil13}. 

Having introduced the continuous and atomistic models and their relation in Section \ref{section:Models}, we devote Section \ref{section:Stability} to a thorough discussion of stability. In the continuous context, the basic stability assumption is the classical Legendre-Hadamard condition. While still necessary for stability of the atomistic system, the Legendre-Hadamard condition in general will not be sufficient to rule out relaxation effects due to microstructure on the atomistic scale. We introduce a suitable atomistic stability constant, investigate its basic properties and also relate it to the Legendre-Hadamard constant in the long wave-length limit. Our discussion is motivated by the results in \cite{hudsonortner}. In particular, we prove that the stability constant is determined in the many particle limit, is thus independent of $\Omega$ and even is equivalent to, though different from, the constant in \cite{hudsonortner}. This allows us to analytically compare atomistic and continuous stability in two-dimensional model cases. 

As a result of independent interest, we are able to describe the onset of instability in (generalized versions of) the Friesecke-Theil mass spring model. In \cite{friesecketheil} Friesecke and Theil noted that the Cauchy-Born rule might fail due to a period doubling shift relaxation if the mismatch of equilibria of the two types of springs within the model is sufficiently large. Our analysis shows that indeed the lattice is stable on all scales up to precisely the point at which period doubling shift relaxations occur. This on the one hand gives a precise description of the stability region in terms of the mismatch parameter. On the other hand it suggests that stability is lost at the critical vaule due to period doubling shift relaxed configurations forming. 

Additionally, we give a new formula in the spirit of the Legendre-Hadamard condition and also provide easier sufficient conditions for stability to show that our atomistic stability assumption is satisfied by a large class of atomistic potentials.

In Section \ref{section:Solvingthecontinuousequations} we shortly discuss the existence of solutions of the boundary value problem in continuous elasticity for small body forces and boundary values in the vicinity of a stable affine deformation under fairly mild regularity assumptions. This is fairly standard and is based on an infinite dimensional implicit function theorem. 

With these preparations we state and prove our main result Theorem \ref{thm:atomisticstatictheorem} in Section \ref{section:ConvergenceToAtomistic}. Similarly as in \cite{emingstatic,ortnertheil13} our goal is to find an atomistic solution in the vicinity of a continuous solution of the associated Cauchy-Born system by observing that the latter is an approximate solution of the discrete system, where now we also have to account for the additional boundary data. To this end, we begin by formulating a quantitative version of the implicit function theorem with a small parameter. While tailored to our application to a singularly perturbed problem, such a result appears to have a wider range of applicability as we discuss at the end of this introduction. From a technical point of view, the main point is then to obtain sufficiently strong residual estimates of the discrete operator acting on the continuous solution. We remark that these estimates cannot be obtained by mere Taylor expansion, but additionally require a subtle regularization on the atomistic scale (cf.\ Propositions \ref{prop:approximation1} and \ref{prop:approximation2}) and cancellation due to a suitable lattice interaction symmetry. We finally conclude by proving Theorem \ref{thm:atomisticstatictheorem}. 

In view of limitations and possible extensions of our work, a natural question is if analogous results can be obtained also in the dynamic setting. Based on our analysis of the static case, this problem will be addressed by the first author in the forthcoming paper \cite{braun16dynamic}, cf.\ also \cite{braunphdthesis}.

Let us also reconsider the still open problem of general atomistic boundary data against the above detailed background on our approach. While in the bulk it is plausible that the Cauchy-Born rule is still approximately true, the situation is--as mentioned--very different close to the surface. In the outermost layers one expects surface relaxation effects. E.g., in case of a free boundary one expects the gradients to have an error of $\mathcal{O}(1)$ while oscillating on the scale $\mathcal{O}(\eps)$. Even though this does not effect the highest order of the energy, it does mean that the Cauchy-Born approximation leaves a residual in the equations that does not vanish as $\eps \to 0$, e.g., in any $L^p$-norm. This makes it much more difficult to find exact solutions to the equations with asymptotically equal bulk behavior. A precise and rigorous mathematical treatment of surface relaxation effects is currently still out of reach. The best known result so far appears to be \cite{theilsurface11}, which gives the correct asymptotics of the surface energy in the limit of vanishing mismatches in the potentials. But even if one were to establish a full characterization of the surface energy, this would still be just a first step towards describing exact solutions of the equations.

We conclude this introduction with a short remark on the more general applicability of the quantitative implicit function theorem, Theorem \ref{thm:quantitativeimplicitfunctiontheorem}. By way of example, let us just touch on one quite different field in which our formulation might have some interest: the area of computer assisted proofs. There, a typical problem can be to find solutions with certain properties to a nonlinear PDE. One starts by finding an approximate solution numerically without the need for a convergence proof. Then, one establishes rigorous bounds on the necessary quantities and concludes with a quantitative existence result that there is indeed a true solution in the vicinity. Such arguments can also be used to establish multiplicity or even uniqueness. Compare \cite{lessardvandenberg15} for a short survey on recent progress. Our specific version of the quantitative existence result then allows for the dependence on additional parameters in small balls. At least in principle, a covering argument directly extends this to parameters in a given compact set which is even allowed to be infinite dimensional. In practice, due to computational restrictions, this is typically only realistic for parameters in a given finite dimensional bounded set.

\section{The Models}\label{section:Models}
\subsection{The Continuous Model}
We consider a bounded, open reference set $\Omega \subset \R^d$, deformations $y \colon \Omega \to \R^d$, a Borel function $W_{\rm cont} \colon \R^{d \times d} \to (-\infty,\infty]$ which is bounded from below, a body force $f \in L^2(\Omega; \R^d)$, a boundary datum $g \in H^1(\Omega ; \R^d)$ and the deformation energy
\[
E(y;f) = \int\limits_\Omega W_{\rm cont}(\nabla y (x)) - y(x)f(x) \,dx.
\]
We are now interested in finding local minimizers of this energy (in a suitable topology) constraint to $y$ having boundary values $g$. In a sufficiently regular setting, these are (stable) solutions of the corresponding Euler-Lagrange equations
\[ \left\{ \begin{array}{r c l l}
      -\divo (DW_{\rm cont}(\nabla y (x))) &=& f(x) &  \text{in}\ \Omega, \\
      y(x) & = & g(x) & \text{on}\ \partial\Omega.
\end{array} \right. \]

We want the assumptions on $W_{\rm cont}$ to be weak enough to include, e.g., Lennard-Jones-type interactions. Therefore, we should not assume global (quasi-)convexity or growth at infinity and $W_{\rm cont}$ should be allowed to have singularities. Of course, under such weak assumptions we cannot hope to solve the problem for all $f$ and $g$. Instead, we will look at a stable affine reference deformation $y_{A_0}(x) = A_0 x$ with gradient $A_0 \in \R^{d \times d}$ and show that for all $f$ small enough and all $g$ close enough to the reference deformation there is a unique deformation close to $y_{A_0}$ that solves the problem. Here, stability is yet to be defined.

\subsection{The Atomistic Model}

Let us first fix some notation.
We consider the reference lattice $ \eps \Z^d$, where $\eps > 0$ is the lattice spacing. Up to a set of measure zero, we partition $\R^d$ into the cubes $\{z\} + \big(-\frac{\eps}{2},\frac{\eps}{2}\big)^d$ with $z \in \eps \Z^d$. Given $x \in \R^d$, not in the neglected set of measure zero, we let $\hat{x}\in \eps \Z^d$ be the midpoint of the corresponding cube and $Q_\eps(x)$ the cube itself. Furthermore, for certain symmetry arguments we will later use the point $\bar{x}$ defined as the reflection of $x$ at $\hat{x}$.

Now, atomistic deformations are maps $y \colon \Omega \cap \eps \Z^d \to \R^d$. Again, we will look at the reference configuration $y_{A_0}(x) = A_0 x$, meaning that the reference positions of the atoms are $A_0 \Omega \cap \eps A_0 \Z^d$ in the macroscopic domain $A_0 \Omega$. The deformation energy is supposed to result from local finite range atomic interactions. More precisely, there is a finite set $\mathcal{R} \subset \Z^d \backslash \{0\}$ accounting for the possible interactions, for which we will always assume that $\spano_{\Z} \mathcal{R}=\Z^d$ and $\mathcal{R} = - \mathcal{R}$. We then assume that the atoms marked by $x,\tilde{x} \in \eps \Z^d$ can only interact directly if there is a point $z\in \eps \Z^d$ with $x,\tilde{x} \in z+ \eps \mathcal{R}$. Furthermore, we assume our system to be translationally invariant such that the interaction can only depend on the matrix of differences $D_{\mathcal{R},\eps} y (x) = (\frac{y(x+\eps \rho)-y(x)}{\eps})_{\rho \in \mathcal{R}}$ with $x \in \eps \Z^d$, where we already use the natural scaling such that $D_{\mathcal{R},\eps} y_{A_0} (x) = (A_0 \rho)_{\rho \in \mathcal{R}}$ for all $\eps>0$. Our site potential $W_{\rm atom} \colon (\R^d)^\mathcal{R} \to (-\infty,\infty]$ is then assumed to be independent of $\eps$. Compare \cite{BLL:02} for a detailed discussion of this scaling assumption.

As a mild symmetry assumption on $W_{\rm atom}$, we will assume throughout that
\[ W_{\rm atom} (A) = W_{\rm atom} (T(A))\]
for all $A \in (\R^d)^\mathcal{R}$, where
\[T(A)_{\rho} = -A_{-\rho}.\]
This is indeed a quite weak assumption. In a typical situation this just means that we have partitioned the overall energy in such a way, that the site potential is invariant under a point reflection at that atom combined with the natural relabeling.
\begin{lem} \label{lem:symmetry}
If $W_{\rm atom}$ satisfies the symmetry condition and $B \in \R^{d \times d}$, then
\[D^k W_{\rm atom} ((B \rho)_{\rho \in \mathcal{R}})[T(A_1), \dotsc, T(A_k)] = D^k W_{\rm atom} ((B \rho)_{\rho \in \mathcal{R}})[A_1, \dotsc, A_k]\]
whenever these derivatives exist.
\end{lem}
\begin{proof}
This follows directly from
$T((B \rho)_{\rho \in \mathcal{R}}) = (B \rho)_{\rho \in \mathcal{R}}$.
\end{proof}
Letting $R_{\rm max} = \max \{\lvert \rho \rvert \colon \rho \in \mathcal{R}\}$ and $R_0 = \max\{R_{\rm max}, \frac{\sqrt{d}}{4}\}$, the discrete gradient $D_{\mathcal{R},\eps}y$ is surely well-defined on the discrete `semi-interior'
\[\sinto_\eps \Omega = \{x \in \Omega \cap \eps \Z^d \colon \dist(x,\partial \Omega) > \eps R_0\}.\]
The total energy below will be defined by a sum over this set, which is justified by our considering variations only on the discrete interior
\[\into_\eps \Omega = \{x \in \Omega \cap \eps \Z^d \colon \dist(x,\partial \Omega) > 2\eps R_0\},\]
which do not affect the gradients outside the semi-interior, and by prescribing boundary values on the layer $\partial_\eps \Omega = \Omega \cap \eps \Z^d \backslash \into_\eps \Omega$. Now, given a body force $f \colon \eps \Z^d \cap \Omega \to \R^d$ and a boundary datum $g \colon \partial_\eps \Omega \to \R^d$ we define the set of admissible deformations as
\[ \mathcal{A}_\eps (\Omega, g) = \{ y \colon \Omega \cap \eps \Z^d \to \R^d \colon y(x)=g(x) \text{ for all } x \in \partial_\eps \Omega\}\]
and the elastic energy of an atomistic deformation $y$ by
\[
E_\eps(y;f;g) = \left\{ \begin{array}{c l l}
      \eps^d \sum\limits_{x \in \sinto_\eps \Omega} W_{\rm atom}(D_{\mathcal{R},\eps} y (x)) - \eps^d \sum\limits_{x \in \eps \Z^d \cap \Omega} y(x)f(x)  & \text{if}\ y \in \mathcal{A}_\eps (\Omega, g), \\
      \infty & \text{else.}
\end{array} \right.
\]
We remark here that the definition of $R_0$ implies that
\[\Omega_\eps = \bigcup_{z \in \into_\eps \Omega} Q_\eps(z) \subset \Omega\]
which will simplify things later on.

As in the continuous case, our goal is to find local minimizers of the energy which, under suitable assumptions on $W_{\rm atom}$, are (stable) solutions of the corresponding Euler-Lagrange equation
\[ -\divo_{\mathcal{R},\eps}\big( DW_{\rm atom}(D_{\mathcal{R},\eps} y(x))\big) = f(x), \]
with $x \in \into_\eps\Omega$, where $DW_{\rm atom}(M) = \big( \frac{\partial W_{\rm atom}(M)}{\partial M_{i\rho}} \big)_{\substack{1 \leq i \leq d \\ \rho \in \mathcal{R}}}$ for $M = (M_{i\rho})_{\substack{1 \leq i \leq d \\ \rho \in \mathcal{R}}} \in \R^{d \times \mathcal{R}} \cong (\R^d)^\mathcal{R}$ and we write
\[\divo_{\mathcal{R},\eps} M(x) = \sum\limits_{\rho \in \mathcal{R}} \frac{M_\rho (x) - M_\rho(x-\eps \rho)}{\eps}\]
for any $M \colon \Omega \cap \eps \Z^d \to \R^{d \times \mathcal{R}} \cong (\R^d)^\mathcal{R}$.
Of course, there is no reason to hope for existence (or uniqueness) in general. We will also restrict ourselves to `elastic' solutions that are (macroscopically) sufficiently close to some affine lattice. To find such solutions we will look close to continuous solutions. 

\subsection{The Cauchy-Born Rule}
As described in detail in the introduction, it is a fundamental problem to identify the correct $W_{\rm cont}$ that should be taken for the continuous equation so that one can hope for atomistic solutions close by as $\eps$ becomes small enough. The classical ansatz to resolve this question by applying the Cauchy-Born leads to setting $W_{\rm cont} = W_{\rm CB}$, where in our setting the Cauchy-Born energy density has the simple mathematical expression  
\[W_{\rm CB}(A) := W_{\rm atom} ((A\rho)_{\rho \in \mathcal{R}}).\]
In the following we will only consider $W_{\rm cont} = W_{\rm CB}$, where $W_{\rm atom}$ is given. Our main goal is to justify this choice rigorously. 

\section{Stability}\label{section:Stability}
A crucial ingredient for our main theorem, but also for further applications, is the concept of atomistic stability. Here we define the continuous and atomistic stability constants, discuss their properties, and give simple characterizations.
\subsection{Stability Constants}
For a bilinear form $L\in \R^{d \times d \times d \times d} \cong \Bil(\R^{d \times d}) \cong L(\R^{d \times d},\R^{d \times d})$ we will write
\begin{align*}
L[A,B] &=  \sum\limits_{j,k,l,m=1}^d L_{jklm} A_{jk} B_{lm},\\
(L[A])_{jk} &=  \sum\limits_{l,m=1}^d L_{jklm} A_{lm}
\end{align*}
if $A, B \in \R^{d \times d}$, and 
\[ \lvert L \rvert = \sup \{ L[A,B] \colon \lvert A \rvert = \lvert B \rvert = 1 \}.\]
Later we will use a similar notation for higher order tensors.

In our problem $L$ is the tensor in the equation
\[ -\divo (L[\nabla u]) = f, \]
which is the linearization of the continuous equation at the affine deformation $y_{A_0}$ if $L=D^2W_{\rm CB}(A_0)$. The condition that ensures existence, uniqueness and regularity and at the same time ensures that solutions are strict local minimizers of the nonlinear energy, and in that sense stability, is the Legendre-Hadamard condition
\[ \lambda_{\rm LH}(L) = \inf\limits_{\xi, \eta \in \R^d\backslash \{0\}} \frac{L[\xi \otimes \eta,\xi \otimes \eta]}{\lvert \xi \rvert^2 \lvert \eta \rvert^2} >0.\]
It is a well known fact, proven by Fourier transformation and a cutoff argument, that
\[ \lambda_{\rm LH}(L) = \inf\limits_{u \in H^1_0(U;\R^d) \backslash \{0\}} \frac{\int_U L[\nabla u (x), \nabla u(x)] \,dx}{\int_U \lvert \nabla u (x) \rvert^2 \,dx}\]
for any open, nonempty $U \subset \R^d$. This is the same as saying that quasiconvexity and rank-one-convexity are equivalent for quadratic densities with constant coefficients, in this case for $L-\lambda \Id$. The result is standard and can be found in the literature, e.g.\ \cite[Thm.~5.25]{dacorogna}.
We also introduce a modified version that is equivalent but more adapted to the atomistic norms we will use in the following:
\[ \tilde{\lambda}_{\rm LH}(L) = \inf\limits_{\xi, \eta \in \R^d\backslash \{0\}} \frac{L[\xi \otimes \eta,\xi \otimes \eta]}{\lvert \xi \rvert^2 \sum\limits_{\rho \in \mathcal{R}} (\rho \eta)^2} >0.\]
Since $\spano_\Z \mathcal{R} = \Z^d$, we have $\spano_\R \mathcal{R} = \R^d$ and there are $C_1,C_2>0$ such that
\[ C_1 \lvert \eta \rvert^2 \leq \sum\limits_{\rho \in \mathcal{R}} (\rho \eta)^2 \leq C_2 \lvert \eta \rvert^2.\]
Hence,
\[C_1 \tilde{\lambda}_{\rm LH}(L) \leq \lambda_{\rm LH}(L) \leq C_2 \tilde{\lambda}_{\rm LH}(L)\]
and, in particular, $\tilde{\lambda}_{\rm LH}(L) > 0$ if and only if $\lambda_{\rm LH}(L)>0$.

In the atomistic setting we have tensors on higher dimensional spaces of the type $K\in \R^{ \{1, \dotsc, d\}\times \mathcal{R} \times \{1, \dotsc, d\}\times \mathcal{R}} \cong \Bil(\R^{\{1, \dotsc, d\}\times \mathcal{R}}) \cong L(\R^{\{1, \dotsc, d\}\times \mathcal{R}},\R^{\{1, \dotsc, d\}\times \mathcal{R}})$. Note that with each such $K$ we can associate a tensor of the form $L\in \R^{d \times d \times d \times d}$ by
\[ L[A,B] = K[(A\rho)_{\rho \in \mathcal{R}},(B\rho)_{\rho \in \mathcal{R}}].\]
In our equations we will consider $K=D^2W_{\rm atom}((A_0\rho)_{\rho \in \mathcal{R}})$, which then corresponds to $L=D^2W_{\rm CB}(A_0)$. It turns out that we need a stronger condition for existence and the local minimizing property in the atomistic case. We define
\[ \lambda_\eps(K,\Omega) = \inf\limits_{\substack{
            y \in \mathcal{A}_\eps(\Omega,0)\\
            y \neq 0}}
            \frac{ \eps^d \sum\limits_{x \in \sinto_\eps \Omega} K[D_{\mathcal{R},\eps} y (x),D_{\mathcal{R},\eps} y (x)]}{\eps^d \sum\limits_{x \in \sinto_\eps \Omega} \lvert D_{\mathcal{R},\eps} y (x)\rvert^2}.\]
Now by atomistic stability we mean that
\[ \lambda_{\rm atom}(K,\Omega) = \inf_{\eps>0} \lambda_\eps(K,\Omega) > 0.\]
We will first show that $\lambda_{\rm atom}$ is in fact independent of $\Omega$ and is equivalently given by the minimization of periodic problems. This can be done in the spirit of a thermodynamical limit argument. 
Let us consider
\[\mathcal{B}_{0,N} = \{ y \colon \{0, \dotsc, N\}^d \to \R^d \colon y(z)=0, \text{ if } \dist(z, \partial (0,N)^d) \leq 2 R_0  \}  \]
and
\[\mathcal{B}_{\text{per},N} = \{ y \colon \{0, \dotsc, N\}^d \to \R^d \colon y \text{ is } [0,N)^d\text{-periodic.}  \}.\]
Whenever necessary we will consider these functions to be $[0,N)^d$-periodically extended to $\Z^d$.

Let us define
\[ \mu_{0,N} = \inf\limits_{\substack{
            y \in \mathcal{B}_{0,N}\\
            y \neq 0}}
            \frac{ \sum\limits_{x \in \{0, \dotsc, N\}^d} K[D_{\mathcal{R},1} y (x),D_{\mathcal{R},1} y (x)]}{ \sum\limits_{x \in \{0, \dotsc, N\}^d} \lvert D_{\mathcal{R},1} y (x)\rvert^2}\]
and
\[ \mu_{\text{per},N} = \inf\limits_{\substack{
            y \in \mathcal{B}_{\text{per},N}\\
            y \text{ not constant}}}
            \frac{ \sum\limits_{x \in \{0, \dotsc, N-1\}^d} K[D_{\mathcal{R},1} y (x),D_{\mathcal{R},1} y (x)]}{ \sum\limits_{x \in \{0, \dotsc, N-1\}^d} \lvert D_{\mathcal{R},1} y (x)\rvert^2}.\]
In this definition we used that $D_{\mathcal{R},1} y (x)= 0$ for all $x \in \Z^d$ implies that $y$ is constant since $\spano_\Z \mathcal{R} = \Z^d$.
Obviously, $\mu_{0,N}$ is nonincreasing and $- \lvert K \rvert \leq \mu_{0,N} \leq  \lvert K \rvert$ for $N$ sufficiently large. Hence, $\mu_{0,N} \to \mu_0 \in [- \lvert K \rvert, \lvert K \rvert]$, where $\mu_0 = \inf_N \mu_{0,N}$.
\begin{prop} \label{prop:stabilityzeroorperiodic}
We have $\mu_{\text{per},N} \to \mu_0$ as $N \to \infty$ and $\mu_0 = \inf_N \mu_{\text{per},N}$.
\end{prop}
\begin{proof}
It is clear that $\mu_{\text{per},N} \leq  \mu_{0,N}$ for all $N$. Now let $\delta>0$ and $N \in \N$ with $N \geq 6 R_0$. Take a nonconstant $y \in \mathcal{B}_{\text{per},N}$ such that
\[ \sum\limits_{x \in \{0, \dotsc, N-1\}^d} K[D_{\mathcal{R},1} y (x),D_{\mathcal{R},1} y (x)] \leq (\mu_{\text{per},N} + \delta) \sum\limits_{x \in \{0, \dotsc, N-1\}^d} \lvert D_{\mathcal{R},1} y (x)\rvert^2.\]
Now we consider $\tilde{y} = \eta y \in \mathcal{B}_{0,MN}$, where $M \in \N$, $M\geq 3$, and $\eta \in C^\infty(\R^d;[0,1])$ such that $\eta(x) = 1$ whenever $ \dist(x, \big((0,MN)^d\big)^c) \geq 4 R_0$, $\eta(x) = 0$ whenever $ \dist(x, \big((0,MN)^d\big)^c) \leq 2 R_0$ and $\lvert \nabla \eta(x) \rvert \leq \frac{1}{R_0}$ for all $x$. A short calculation gives
\[ \lvert D_{\mathcal{R},1} \tilde{y}(x) \rvert \leq C(d,\lvert\mathcal{R}\rvert,R_0) \lVert y \rVert_\infty \]
for all $x$. Using this we can estimate
\begin{align*} \sum\limits_{x \in \{0, \dotsc, MN-1\}^d} K[D_{\mathcal{R},1} \tilde{y} (x)&,D_{\mathcal{R},1} \tilde{y} (x)] \leq (M-2)^d(\mu_{\text{per},N} + \delta) \sum\limits_{x \in \{0, \dotsc, N-1\}^d} \lvert D_{\mathcal{R},1} y (x)\rvert^2\\
&\quad+ C(d, \lvert\mathcal{R}\rvert,R_0) N^d M^{d-1} \lvert K \rvert \lVert y \rVert_\infty^2 \\
&\leq (\mu_{\text{per},N} + \delta) \sum\limits_{x \in \{0, \dotsc, MN-1\}^d} \lvert D_{\mathcal{R},1} \tilde{y} (x)\rvert^2\\
&\quad+ C(d, \lvert\mathcal{R}\rvert,R_0) N^d M^{d-1} ( \lvert K \rvert + \lvert \mu_{\text{per},N} \rvert + \delta) \lVert y \rVert_\infty^2.
\end{align*}
But for $M$ large enough we have
\[ C(d, \lvert\mathcal{R}\rvert,R_0) N^d M^{d-1} ( \lvert K \rvert + \lvert \mu_{\text{per},N} \rvert + \delta) \lVert y \rVert_\infty^2 \leq \delta \sum\limits_{x \in \{0, \dotsc, MN-1\}^d} \lvert D_{\mathcal{R},1} \tilde{y} (x)\rvert^2. \]
Therefore,
\[\mu_0 \leq \mu_{0,MN} \leq  \mu_{\text{per},N} + 2\delta.\]
The restriction $N \geq 6 R_0$ is not problematic when we take the infimum since
\[\mu_{\text{per},jN} \leq \mu_{\text{per},N}\]
for all $j \in \N$.
\end{proof}
\begin{prop} \label{prop:stabilitydomainindependent}
For all open, bounded, nonempty $\Omega \subset \R^d$ we have
\[\lambda_{\rm atom}(K,\Omega)= \lim\limits_{\eps \to 0^+} \lambda_\eps (K, \Omega) = \mu_0.\]
\end{prop}
\begin{proof}
Take $z_1, z_2 \in \R^d$ and $0<a_1<a_2$ such that
\[ \{z_1\} + [0,a_1]^d \subset \Omega \subset \{z_2\} + ( 0,a_2)^d. \]
Now, define $(z_{1,\eps})_i = \lceil\frac{(z_1)_i}{\eps}\rceil\eps$, $(z_{2,\eps})_i = \lfloor\frac{(z_1)_i}{\eps}\rfloor \eps$, $N_{1,\eps} = \lfloor\frac{a_1}{\eps}\rfloor - 1$ and $N_{2,\eps} = \lceil\frac{a_2}{\eps}\rceil + 1$. Then,
\[ z_{1,\eps}+ \eps \{0, \dotsc, N_{1,\eps}\}^d \subset \Omega \cap \eps \Z^d \subset z_{2,\eps}+ \eps \{0, \dotsc, N_{2,\eps}\}^d.\]
Given $y \in \mathcal{B}_{0,N_{1,\eps}}$ we can set
\[ \tilde{y}(z_{1,\eps} + \eps v) = \eps y(v)\]
for $v \in \{0, \dotsc, N_{1,\eps}\}^d$ and $\tilde{y}(x)=0$ else. Then $\tilde{y} \in \mathcal{A}_\eps(\Omega,0)$,
\[\sum\limits_{x \in \sinto_\eps \Omega} K[D_{\mathcal{R},\eps} \tilde{y} (x),D_{\mathcal{R},\eps} \tilde{y} (x)] = \sum\limits_{x \in \{0, \dotsc, N_{1,\eps}-1\}^d} K[D_{\mathcal{R},1} y (x),D_{\mathcal{R},1} y (x)] \]
and
\[\sum\limits_{x \in \sinto_\eps \Omega} \lvert D_{\mathcal{R},\eps} \tilde{y} (x) \rvert^2 = \sum\limits_{x \in \{0, \dotsc, N_{1,\eps}-1\}^d} \lvert D_{\mathcal{R},1} y (x)\rvert^2.\]
Hence,
\[\mu_{0, \lfloor\frac{a_1}{\eps}\rfloor - 1} \geq \lambda_\eps (K, \Omega).\]
Similarly,
\[\mu_{0, \lceil\frac{a_2}{\eps}\rceil + 1} \leq \lambda_\eps (K, \Omega).\]
This holds for all $\eps>0$, if we set $\mu_{0, - 1}= \infty$.
Therefore,
\[ \lim\limits_{\eps \to 0} \lambda_\eps (K, \Omega) = \inf\limits_{\eps > 0} \lambda_\eps (K, \Omega) = \mu_0.\]
\end{proof}
Because of Proposition \ref{prop:stabilitydomainindependent} we can now just write $\lambda_{\rm atom}(K)$. We will also abuse the notation by writing $\lambda_{\rm LH}(K)$ for the stability constant of the corresponding $\R^{d \times d \times d \times d}$ tensor. In the case $K=D^2W_{\rm atom}((A_0\rho)_{\rho \in \mathcal{R}})$ and $L = D^2W_{\rm CB}(A_0)$ with $A_0 \in \R^{d \times d}$ we also write $\lambda_{\rm atom}(A_0)$ and $\lambda_{\rm LH}(A_0)$ and suppress the dependency on $W_{\rm atom}$. In the same way we will write $\lambda_{\rm atom}(A_0)$ instead of $\lambda_{\rm atom}(D^2W_{\rm atom}(A_0))$ if $A_0 \in \R^{d \times \mathcal{R}}$.  

For the dependency on $K$ of $\lambda_{\rm atom}$ we record the following elementary observation.
\begin{prop} \label{prop:stabilityconstantcontinuous}
Given tensors $K,\tilde{K}$, we have
\[\lvert \lambda_{\rm atom}(K)- \lambda_{\rm atom}(\tilde{K}) \rvert \leq \lvert K - \tilde{K} \rvert.\]
In particular, if $W_{\rm atom} \in C^2(V)$, $V$ open, then
\[\{A \in V \colon \lambda_{\rm atom}(A) >0\}\]
is open as well.
\end{prop}
\begin{proof}
This is straightforward. Just use
\begin{align*}
&\bigg\lvert \sum\limits_{x \in \{0, \dotsc, N\}^d} K[D_{\mathcal{R},1} y (x),D_{\mathcal{R},1} y (x)] - \sum\limits_{x \in \{0, \dotsc, N\}^d} \tilde{K}[D_{\mathcal{R},1} y (x),D_{\mathcal{R},1} y (x)] \bigg\rvert\\
&\leq \lvert K - \tilde{K} \rvert \sum\limits_{x \in \{0, \dotsc, N\}^d} \lvert D_{\mathcal{R},1} y (x) \rvert^2 
\end{align*}
\end{proof}

\subsection{Representation Formulae}
Combining Proposition \ref{prop:stabilityzeroorperiodic} and Proposition \ref{prop:stabilitydomainindependent} we are now basically in the setting of the stability discussion in \cite{hudsonortner}. We include the most important points here to stay self-contained and, more importantly, to provide a new, more intuitive characterization of the stability constant and sufficient criteria for stability that allow for a direct application in interesting situations. 

Remark that we use an $h^1$-Norm based on difference quotients, while the authors in \cite{hudsonortner} use a Fourier norm. Therefore, the stability constants will be different, but of course everything remains equivalent. While the Fourier norm makes the connection to the continuum case slightly more direct, our approach has the advantage that it is considerably easier to check if the atomistic stability condition holds true in specific situations making it possible to rigorously discuss relatively simple examples as will be detailed in the Section \ref{sec:examplesstability}.

We will write
\[ Q_N = \{0, 1, \dotsc, N-1\}^d\]
and
\[ \hat{Q}_N = \Big\{0, \frac{2 \pi}{N}, \dotsc, \frac{2 \pi (N-1)}{N}\Big\}^d\]
for the dual group. Given $y \colon Q_N \to \C$, its Fourier transformation is defined by $\hat{y} \colon \hat{Q}_N \to \C$ with
\[ \hat{y}(k) = \frac{1}{N^d} \sum\limits_{x \in Q_N} y(x)e^{-ixk}.\]
We have
\[ \sum\limits_{x \in Q_N} e^{ix(k-k')} = N^d \delta_{k,k'}\]
for all $k,k' \in \hat{Q}_N$ and
\[ \sum\limits_{k \in \hat{Q}_N} e^{i(x-x')k} = N^d \delta_{x,x'}\]
for all $x,x' \in Q_N$. Therefore,
\[ y(x) = \sum\limits_{k \in \hat{Q}_N} \hat{y}(k)e^{ixk}\]
for all $x\in Q_N$.

In the following we will often assume that $K_{j \rho l \sigma} = K_{l \sigma j \rho}$ which is automatically satisfied if $K$ is the second derivative of a potential. Furthermore, we will sometimes assume $K_{j \rho l \sigma} = K_{j (-\rho) l (-\sigma)}$, which is satisfied in our models because of the symmetry condition and Lemma \ref{lem:symmetry}. 

In Fourier space the problem is in diagonal form.
\begin{prop} \label{prop:diagonalizationinFourierspace}
Assume that $K_{j \rho l \sigma} = K_{l \sigma j \rho}$ for all $j,l,\rho,\sigma$.
Now, given $y \colon Q_N \to \R^d$ periodically extended to $\Z^d$, we have
\[\sum\limits_{x\in Q_N} K[D_{\mathcal{R},1}y(x), D_{\mathcal{R},1}y(x)] = N^d \sum\limits_{k \in \hat{Q}_N}\hat{y}(k)^T H(k) \overline{\hat{y}(k)},\]
where
\[H(k)_{jl}= \sum\limits_{\rho,\sigma \in \mathcal{R}} K_{j \rho l \sigma} \big( \cos(\rho k) -1 +i \sin(\rho k) \big)\big( \cos(\sigma k) -1 -i \sin(\sigma k) \big).\]
In particular, $H(k)$ is hermitian for all $k$, $H$ is $[0,2\pi)^d$-periodic and $\overline{H(k)}=H(-k)$ for all $k$.

Furthermore, if $K$ additionally satisfies $K_{j \rho l \sigma} = K_{j (-\rho) l (-\sigma)}$, then
\[H(k)_{jl}= \sum\limits_{\rho,\sigma \in \mathcal{R}} K_{j \rho l \sigma} \big( (\cos(\rho k) -1)(\cos(\sigma k) -1) + \sin(\rho k)\sin(\sigma k) \big).\]
In particular, $H(k) \in \R^{d \times d}_{\rm sym}$ for all $k$.
\end{prop}
\begin{proof}
\begin{align*}
\sum\limits_{x\in Q_N} &K[D_{\mathcal{R},1}y(x), D_{\mathcal{R},1}y(x)]\\
&=  \sum\limits_{x,j,l,\rho,\sigma,k,k'} K_{j \rho l \sigma} (e^{i \rho k} - 1)(e^{-i \sigma k'} - 1) e^{ix(k-k')} \hat{y}_j(k) \overline{\hat{y}_l(k')} \\
&= N^d \sum\limits_{j,l,\rho,\sigma,k} K_{j \rho l \sigma} \big( \cos(\rho k) -1 +i \sin(\rho k) \big)\big( \cos(\sigma k) -1 -i \sin(\sigma k) \big) \hat{y}_j(k) \overline{\hat{y}_l(k)}.
\end{align*}
Everything else follows easily since
\begin{align*}
( \cos(\rho k) -1 &+i \sin(\rho k) )( \cos(\sigma k) -1 -i \sin(\sigma k) )\\
&= ( \cos(\rho k) -1)( \cos(\sigma k) -1) + \sin(\rho k)\sin(\sigma k)\\
&\quad + i (\cos(\sigma k) -1) \sin(\rho k)- i (\cos(\rho k) -1) \sin(\sigma k).
\end{align*}
\end{proof}
\begin{prop} \label{prop:stabilityinFourierspace1}
Assume that $K_{j \rho l \sigma} = K_{l \sigma j \rho}$ for all $j,l,\rho,\sigma$. Then,
\[\mu_{\text{per},N}= \min \bigg\{\frac{h(k)}{\sum_{\rho \in \mathcal{R}} (1-\cos(\rho k))^2+ \sin^2(\rho k)} \colon k \in \hat{Q}_N \backslash \{0\} \bigg\} ,\]
where $h(k)$ is the smallest eigenvalue of $H(k)$.
\end{prop}
\begin{proof}
First of all note that, for $k \in \hat{Q}_N$, $\sum_{\rho \in \mathcal{R}} (1-\cos(\rho k))^2+ \sin^2(\rho k) =0$ is equivalent to $k\rho \in 2 \pi \Z$ for all $\rho \in \mathcal{R}$ or, equivalently, for all $\rho \in \Z^d$ since $\spano_\Z \mathcal{R} = \Z^d$. This is the case if and only if $k=0$. Now set
\[\mu_{\text{F},N}= \min \bigg\{\frac{h(k)}{ \sum_{\rho \in \mathcal{R}} (1-\cos(\rho k))^2+ \sin^2(\rho k)} \colon k \in \hat{Q}_N\backslash \{0\} \bigg\}.\]
Given $y \colon Q_N \to \R^d$, periodically extended, we have
\begin{align*}
\sum\limits_{x\in Q_N} &K[D_{\mathcal{R},1}y(x), D_{\mathcal{R},1}y(x)]\\
&= N^d \sum\limits_{k \in \hat{Q}_N}\hat{y}(k)^T H(k) \overline{\hat{y}(k)}\\
&\geq N^d \sum\limits_{k \in \hat{Q}_N} h(k) \lvert \hat{y}(k) \rvert^2 \\
&\geq  \mu_{\text{F},N} N^d \sum\limits_{k \in \hat{Q}_N} \sum_{\rho \in \mathcal{R}} \lvert \hat{y}(k) \rvert^2 \big((1-\cos(\rho k))^2+ \sin^2(\rho k) \big) \\
&= \mu_{\text{F},N} \sum\limits_{x\in Q_N} \lvert D_{\mathcal{R},1}y(x) \rvert^2,
\end{align*}
where we used Proposition \ref{prop:diagonalizationinFourierspace} for $K$ and $\tilde{K}_{j \rho l \sigma} = \delta_{jl} \delta_{\rho \sigma}$. This proves $\mu_{\text{per},N} \geq \mu_{\text{F},N}$.
For the opposite inequality take $k_0 \in \hat{Q}_N\backslash \{0\}$ such that
\[h(k_0) = \mu_{\text{F},N} \sum_{\rho \in \mathcal{R}} (1-\cos(\rho k))^2+ \sin^2(\rho k).\]
Let $v_0$ be a corresponding eigenvector and $k_1 \in \hat{Q}_N$ be the unique vector such that $k_0+k_1 \in 2\pi \Z^d$. In the case $k_0=k_1$, take $v_0$ real. We define
\[ y(x) = \overline{v}_0 e^{ik_0x}+ v_0 e^{ik_1x}.\]
For $x \in Q_N$ we have $y(x) = 2 \Reo v_0 \cos(k_0 x) + 2 \Imo v_0 \sin(k_0 x)$, which is real, $[0,N)^d$-periodic and nonconstant. We calculate
\begin{align*}
\sum\limits_{x\in Q_N} &K[D_{\mathcal{R},1}y(x), D_{\mathcal{R},1}y(x)]\\
&= N^d \sum\limits_{k \in \hat{Q}_N}\hat{y}(k)^T H(k) \overline{\hat{y}(k)}\\
&= 2 N^d h(k_0) \lvert v_0 \rvert^2 (1+ \delta_{k_0 k_1}) \\
&= \mu_{\text{F},N} N^d \sum\limits_{k \in \hat{Q}_N} \sum_{\rho \in \mathcal{R}} \lvert \hat{y}(k) \rvert^2 \big((1-\cos(\rho k))^2+ \sin^2(\rho k) \big) \\
&= \mu_{\text{F},N} \sum\limits_{x\in Q_N} \lvert D_{\mathcal{R},1}y(x) \rvert^2.
\end{align*}
Therefore, $\mu_{\text{per},N} \leq \mu_{\text{F},N}$.
\end{proof}
In the limit $N \to \infty$ we get the following result:
\begin{thm} \label{thm:stability}
Assume that $K_{j \rho l \sigma} = K_{l \sigma j \rho}$ for all $j,l,\rho,\sigma$. Then
\begin{align*}
\lambda_{\rm atom}(K)&= \inf \bigg\{\frac{h(k)}{\sum_{\rho \in \mathcal{R}} (1-\cos(\rho k))^2+ \sin^2(\rho k)} \colon k \in [0,2\pi)^d\backslash \{0\} \bigg\},\\
\tilde{\lambda}_{\rm LH}(K)&= \lim\limits_{s \to 0^+} \inf \bigg\{\frac{h(k)}{\sum_{\rho \in \mathcal{R}} (1-\cos(\rho k))^2+ \sin^2(\rho k)} \colon k \in (-s,s)^d\backslash \{ 0 \} \bigg\},
\end{align*}
where $h(k)$ is the smallest eigenvalue of $H(k)$. In particular, atomistic stability implies the Legendre-Hadamard condition.
\end{thm}
\begin{proof}
Set
\[ \mu_{\text{F}} = \inf \bigg\{\frac{h(k)}{\sum_{\rho \in \mathcal{R}} (1-\cos(\rho k))^2+ \sin^2(\rho k)} \colon k \in [0,2\pi)^d\backslash \{0\} \bigg\}\]
By Proposition \ref{prop:stabilityinFourierspace1} we have $\mu_{\text{per},N} \geq \mu_{\text{F}}$ and thus $\lambda_{\rm atom}(K) \geq \mu_{\text{F}}$. For the opposite inequality let $M>\mu_{\text{F}}$. Now, take $k_0 \in [0,2\pi)^d\backslash \{0\}$ such that
\[\frac{h(k_0)}{\sum_{\rho \in \mathcal{R}} (1-\cos(\rho k_0))^2+ \sin^2(\rho k_0)} < M.\]
By continuity of $h$, we can find an $N \in \N$ and a $k_1 \in \hat{Q}_N$ such that
\[\frac{h(k_1)}{\sum_{\rho \in \mathcal{R}} (1-\cos(\rho k_1))^2+ \sin^2(\rho k_1)} < M.\]
Therefore, $\lambda_{\rm atom}(K) \leq \mu_{\text{per},N} < M$.

Now, let $\eta \in \R^d$ with $\lvert \eta \rvert =1$ and $0<\tau \leq 1$. Then,
\[ \big\lvert (1-\cos(\rho \eta \tau))^2+ \sin^2(\rho \eta \tau)  - \tau^2  (\rho \eta)^2 \big\rvert \leq C \tau^4 \]
and for $\xi \in \C^d$ with $\lvert \xi \rvert = 1$
\[ \big\lvert \xi^T H(\eta \tau)\overline{\xi} - \tau^2 K[\xi \otimes (\rho \eta)_{\rho \in \mathcal{R}}, \overline{\xi} \otimes (\rho \eta)_{\rho \in \mathcal{R}}] \big\rvert \leq C \tau^3.\]
This implies
\[ \big\lvert h(\eta \tau) - \min\limits_{\substack{\xi \in \C^d \\ \lvert \xi \rvert = 1 \\}}\tau^2 K[\xi \otimes (\rho \eta)_{\rho \in \mathcal{R}}, \overline{\xi} \otimes (\rho \eta)_{\rho \in \mathcal{R}}] \big\rvert \leq C \tau^3.\]
Furthermore, for all $\eta$ as above we have
\[0<c \leq \sum_{\rho \in \mathcal{R}} (\rho \eta)^2 \leq C\]
and
\[ \Big\lvert \min\limits_{\substack{\xi \in \C^d \\ \lvert \xi \rvert = 1 \\}} K[\xi \otimes (\rho \eta)_{\rho \in \mathcal{R}}, \overline{\xi} \otimes (\rho \eta)_{\rho \in \mathcal{R}}] \Big\rvert \leq C.\]
Thus, for $\tau$ small enough we also know that
\[\sum_{\rho \in \mathcal{R}} (1-\cos(\rho \eta \tau))^2+ \sin^2(\rho \eta \tau) \geq \frac{c \tau^2}{2}.\]
Due to the symmetry of $K$ we have
\[K[\xi \otimes b, \overline{\xi} \otimes b] = K[\Reo \xi \otimes b, \Reo \xi \otimes b] + K[\Imo \xi \otimes b, \Imo \xi \otimes b] \]
for all $\xi \in \C^d$ and $b \in \R^\mathcal{R}$. In particular,
\[\min\limits_{\substack{\xi \in \C^d \\ \lvert \xi \rvert = 1 \\}} K[\xi \otimes (\rho \eta)_{\rho \in \mathcal{R}}, \overline{\xi} \otimes (\rho \eta)_{\rho \in \mathcal{R}}] = \min\limits_{\substack{\xi \in \R^d \\ \lvert \xi \rvert = 1 \\}} K[\xi \otimes (\rho \eta)_{\rho \in \mathcal{R}}, \xi \otimes (\rho \eta)_{\rho \in \mathcal{R}}]\]
Combining the above inequalities we get
\[ \Bigg\lvert \frac{h(\eta \tau)}{\sum_{\rho \in \mathcal{R}} (1-\cos(\rho \eta \tau))^2+ \sin^2(\rho \eta \tau)} - \frac{\min\limits_{\xi \in \R^d, \lvert \xi \rvert = 1} K[\xi \otimes (\rho \eta)_{\rho \in \mathcal{R}}, \xi \otimes (\rho \eta)_{\rho \in \mathcal{R}}]}{\sum_{\rho \in \mathcal{R}} (\rho \eta)^2} \Bigg\rvert \leq \frac{4C^2}{c^2} \tau \]
for all $\tau$ small enough and all $\eta$ as above. Therefore,
\[\lim\limits_{\tau \to 0^+} \min\limits_{\lvert \eta \rvert = 1} \frac{h(\eta \tau)}{\sum_{\rho \in \mathcal{R}} (1-\cos(\rho \eta \tau))^2+ \sin^2(\rho \eta \tau)} =  \tilde{\lambda}_{\rm LH}(K)\]
which gives the desired result.
\end{proof}
If $H$ is real we can express $\lambda_{\rm atom}$ in a way that looks quite similar to the definition of $\lambda_{\rm LH}$.
\begin{cor}
Assume that $K_{j \rho l \sigma} = K_{l \sigma j \rho}$ and additionally $K_{j \rho l \sigma} = K_{j (-\rho) l (-\sigma)}$ or $K_{j \rho l \sigma} = K_{l \rho j \sigma}$ for all $j,l,\rho,\sigma$. Then
\begin{align*}
\lambda_{\rm atom}(K)&= \inf \bigg\{\frac{K[\xi \otimes c(k),\xi \otimes c(k)]  + K[\xi \otimes s(k), \xi \otimes s(k)]}{\lvert \xi \rvert^2 (\lvert c(k) \rvert^2 + \lvert s(k) \rvert^2)} \colon\\
&\qquad \xi \in \R^d \backslash \{0\}, k \in [0,2\pi)^d\backslash \{0\} \bigg\},
\end{align*}
where $c(k)_\rho = \cos(\rho k) -1$ and $s(k)_\rho = \sin(\rho k)$.
\end{cor}
The following criterion is strictly weaker but often easier to check.
\begin{cor} \label{cor:easyatomisticstability}
Assume that $K_{j \rho l \sigma} = K_{l \sigma j \rho}$ and additionally $K_{j \rho l \sigma} = K_{j (-\rho) l (-\sigma)}$ for all $j,l,\rho,\sigma$. Let $\lambda_{\rm LH}(K)>0$, $K[\xi \otimes s(k),\xi \otimes s(k)] \geq 0$ for all $\xi,k \in \R^d$ and
\[K[\xi \otimes c(k),\xi \otimes c(k)] \geq  \gamma \lvert \xi \rvert^2 \lvert c(k) \rvert^2\]
for all $\xi,k \in \R^d$ and some $\gamma > 0$. Then $\lambda_{\rm atom}>0$.
\end{cor}  
\begin{proof}
Since $\lambda_{\rm LH}(K)$ and $\tilde{\lambda}_{\rm LH}(K)$ are equivalent, we can use Theorem \ref{thm:stability} to see that that there are some $\tilde{\gamma}, \delta > 0$ such that
\[K[\xi \otimes c(k),\xi \otimes c(k)]  + K[\xi \otimes s(k), \xi \otimes s(k)] \geq \tilde{\gamma}\lvert \xi \rvert^2 (\lvert c(k) \rvert^2 + \lvert s(k) \rvert^2)\]
for all $\xi$ and all $k$ with $\dist (k, 2\pi \Z^d) < \delta$. On the other hand, there is a $C>0$ such that $\lvert s(k) \rvert \leq C \lvert c(k) \rvert$ whenever $\dist (k, 2\pi \Z^d) \geq \delta$. Therefore
 \[K[\xi \otimes c(k),\xi \otimes c(k)]  + K[\xi \otimes s(k), \xi \otimes s(k)] \geq \frac{\gamma}{1+C^2}\lvert \xi \rvert^2 (\lvert c(k) \rvert^2 + \lvert s(k) \rvert^2)\]
 for these $k$ and all $\xi$.
\end{proof}

\begin{rem}
The connection to the formulas in \cite{hudsonortner} is given by
\begin{align*}
4 \sin^2\Big(\frac{z}{2}\Big) &= (\cos(z)-1)^2 + \sin^2(z)\\
2\sin^2\Big(\frac{y}{2}\Big) + 2\sin^2\Big(\frac{z}{2}\Big) - 2\sin^2\Big(\frac{z-y}{2}\Big) &= (\cos(y)-1)(\cos(z)-1) + \sin(y)\sin(z).
\end{align*}
A little bit of calculation shows that the stability constants here and in \cite{hudsonortner} then are actually equivalent (with the minor correction, that most of their sums should actually run over the set $\mathcal{R}-\mathcal{R}$ instead of $\mathcal{R}$).
\end{rem}

\subsection{Examples for Stability} \label{sec:examplesstability}
First of all, let us point out that the general assumptions made in this work are consistent with a large variety of atomic interaction models and lattices. A simple sufficient condition for atomistic stability is the following: 
\begin{prop}
If $W_{\rm atom} \in C^2$ close to $(A_0 \rho)_{\rho \in \mathcal{R}}$, satisfies the symmetry condition and $(A_0 \rho)_{\rho \in \mathcal{R}}$ is a local minimizer of the energy, such that the second derivative in the directions of affine rank-one deformations $((\xi \otimes \eta) \rho)_{\rho \in \mathcal{R}}$ and on the orthogonal complement of all affine deformations is strictly positive. Then $\lambda_{\rm atom} (A_0)>0$.
\end{prop}
\begin{proof}
Just use Corollary \ref{cor:easyatomisticstability} and the fact that $\xi \otimes c(k)$ is orthogonal on affine deformations.
\end{proof}
\begin{rem}
These conditions allow for a large class of frame indifferent interaction models. Examples include the general finite range potentials discussed in \cite{CDKM}. 
\end{rem}
We next want to discuss the connection between atomistic and continuous stability. To do this we will characterize the stability constants in two examples. The examples are two-dimensional to allow for a significantly easier analytical treatment, but the studied effects are expected to be the same in three dimensions.

There is a conjecture, that in certain regimes one has $\lambda_{\rm atom}(A) = \tilde{\lambda}_{\rm LH}(A)$ for a large set of matrices or at least $\lambda_{\rm atom}(A)>0$ if and only if $\lambda_{\rm LH}(A)>0$, compare \cite{hudsonortner}. But so far this has only been proven in certain one-dimensional cases, e.g. in \cite{hudsonortner}. Even more importantly, this is expected to be false in general. In more than one dimension so far this has only been discussed numerically in \cite{hudsonortner}.

First, let us look at a rather simple but multidimensional example where it is possible to analytically prove $\lambda_{\rm atom}(A) = \tilde{\lambda}_{\rm LH}(A)$ for a large set of matrices $A$. To be more precise, we consider uniform contractions and extensions of a triangular lattice where the energy is given by an unspecified pair potential for the nearest neighbors. This means we will look at $d=2$,
\[ M = \begin{pmatrix}
1 & \frac{1}{2}\\ 0 & \frac{\sqrt{3}}{2}
\end{pmatrix}
\]
and consider the linearization at $M(t) = t M$ for $t>0$. Furthermore, $\mathcal{R} = \{ \pm e_1,\pm e_2,\pm (e_2-e_1)\}$ and the interaction is given by
\[ W_{\rm atom}(A) = \frac{1}{2} \sum\limits_{\rho \in \mathcal{R}} V_0(\lvert A_\rho \rvert)\]
with some pair potential $V_0 \in C^2((0,\infty);\R)$.
The Cauchy-Born energy density is then given by
\[W_{\rm CB}(A) = V_0 ( \lvert A_{\cdot 1} \rvert ) + V_0 ( \lvert A_{\cdot 2} \rvert ) + V_0 ( \lvert A_{\cdot 2}-A_{\cdot 1} \rvert ).\]
Direct calculations give
\[K(t)_{j \rho l \sigma} = \delta_{\rho \sigma} \Big( \frac{V_0'(t)}{t} (\delta_{jl} - (M\rho)_j (M \rho)_l) + V_0''(t)(M\rho)_j (M \rho)_l \Big)\]
and, with some more effort,
\begin{align*}
h(t,k) &= 4 \Big(V_0''(t) + \frac{V_0'(t)}{t}\Big)\Big(\sin^2\Big(\frac{k_1}{2}\Big) + \sin^2\Big(\frac{k_2}{2}\Big) + \sin^2\Big(\frac{k_2-k_1}{2}\Big)\Big)\\
&\qquad - 2\sqrt{2} \Big\lvert V_0''(t) - \frac{V_0'(t)}{t}\Big\rvert \bigg(\Big(\sin^2\Big(\frac{k_1}{2}\Big)-\sin^2\Big(\frac{k_2}{2}\Big)\Big)^2\\
&\qquad +\Big(\sin^2\Big(\frac{k_1}{2}\Big)-\sin^2\Big(\frac{k_2-k_1}{2}\Big)\Big)^2+\Big(\sin^2\Big(\frac{k_2-k_1}{2}\Big)-\sin^2\Big(\frac{k_2}{2}\Big)\Big)^2\bigg)^{\frac{1}{2}}.
\end{align*}

The nonlinear minimization problem can be drastically simplified by the substitution $s_1 = \sin(\frac{k_1}{2})$ and $s_2= \sin(\frac{k_2}{2})$. Then, only certain algebraic inequalities have to be shown. A lengthy but not too difficult calculation results in the following characterization. All omitted details can be found in \cite{braunphdthesis}.

\begin{prop}
In the above setting we have
\[
\lambda_{\rm atom}(M(t))=\tilde{\lambda}_{\rm LH}(M(t))= \frac{1}{2} \Big(V_0''(t) + \frac{V_0'(t)}{t}\Big) - \frac{1}{4} \Big\lvert V_0''(t) - \frac{V_0'(t)}{t}\Big\rvert.
\]
\end{prop}
\begin{rem}
If $V_0$ is a standard Lennard-Jones potential, i.e.,
\[ V_0(r) = r^{-12} - 2r^{-6},\]
then $M(t)$ is stable in both senses if and only if \[t \in \Big(0,\sqrt[6]{\frac{19}{10}}\Big),\]
where $\sqrt[6]{\frac{19}{10}} \approx 1.113$.
\end{rem}
\begin{rem}
In the proof the choice of our $h^1$-norm helps to drastically simplify the problem. If one tries to show the equivalent result for the Fourier $h^1$-norm one has to prove a fully nonlinear, nonalgebraic inequality. In the approach above only a few algebraic manipulations are necessary.
\end{rem}

As a second example to actually show the differences between the two notions of stability we want to look at a rectangular lattice with nearest and next-to-nearest neighbor interactions that are not balanced with each other. In \cite{friesecketheil} this problem is discussed in the context of global minimization. We will look at the same setting and give an explicit characterization of our notions of (local) stability. As in \cite{friesecketheil} the instability we find is a ``shift-relaxation'', which corresponds to a period doubling. But we prove even more. We show that there are no macroscopic instabilities at all and that the lattice is stable on all scales up to the point where the instability due to ``shift-relaxations'' occurs. Additionally, since we only require a local analysis, it is easy to extend the example to a quite general class of potentials, as we will describe in more detail at the end. 

We set $\mathcal{R} = \{\pm e_1, \pm e_2, e_1 \pm e_2, -e_1 \pm e_2 \}$ and
\[W_{\rm atom}(A) = \frac{K_1}{4} \sum_{\rho \in \mathcal{R}, \lvert \rho \rvert = 1} (\lvert A_\rho \rvert - a_1)^2 +  \frac{K_2}{4} \sum_{\rho \in \mathcal{R}, \lvert \rho \rvert = \sqrt{2}} (\lvert A_\rho \rvert - a_2)^2\]
for some $a_1,a_2,K_1,K_2>0$. We are now interested in the stability of $A_0 = r^\ast \Id$ with
\[ r^\ast = \frac{K_1 a_1 + \sqrt{2} K_2 a_2}{K_1 + 2 K_2}.\]
In the following let us use the notation
\[\alpha = \frac{a_2}{\sqrt{2}a_1}, \quad \kappa= \frac{K_2}{K_1} \quad \text{and}\ \beta = \frac{1+2\kappa}{1+2\alpha \kappa}.\]
\begin{prop}
In this setting we have 
\[ \tilde{\lambda}_{\rm LH}(r^\ast \Id) = \frac{K_1}{12} \beta \min\{1,2\alpha\kappa\}>0 \]
for all parameter values, while $\lambda_{\rm atom}(r^\ast \Id) >0$ if and only if $\beta <2$, which corresponds to $\alpha\geq \frac{1}{2}$ or $\alpha <\frac{1}{2}$ and $\kappa < \frac{1}{2(1-2\alpha)}$.
\end{prop}
\begin{proof}
Calculating the derivatives we find
\begin{align*}
K_{j\rho l \sigma} &= D^2 W_{\rm atom}((r^\ast \rho)_{\rho \in \mathcal{R}})[e_j \otimes e_\rho, e_l \otimes e_\sigma]\\
&= \delta_{\rho \sigma} \delta_{\lvert \rho \rvert 1}\Big(\delta_{jl} \frac{K_1}{2}\big(1-\frac{a_1}{r^\ast}\big) + \rho_j \rho_l\frac{K_1 a_1}{2 r^\ast}\Big)\\
&\quad + \delta_{\rho \sigma} \delta_{\lvert \rho \rvert \sqrt{2}}\Big(\delta_{jl} \frac{K_2}{2}\big(1-\frac{a_2}{\sqrt{2} r^\ast}\big) + \rho_j \rho_l\frac{K_2 a_2}{4 \sqrt{2} r^\ast}\Big).
\end{align*}
One then proceeds similarly to the last example. All details can again be found in \cite{braunphdthesis}.
\end{proof}

In this example we see that the Legendre-Hadamard stability constant and the atomistic stability constant can be quite different and the parameter regions where we have macroscopic or atomistic stability can be very different as well. In the Fourier characterization it is clear that this difference occurs whenever a system is stable under macroscopic, long wavelength perturbations but not under some perturbation with wavelength on the atomistic scale. In this example, the instability does indeed occur on the atomistic scale and actually corresponds to a period doubling where the wave number is $k=(\pi,\pi)$.

The example is actually much more general than it looks. Given general pair potentials $V_1, V_2 \in C^2(0,\infty)$ as well as an $r^\ast$ with
\[V_1(r^\ast) + \sqrt{2} V_2'(r^\ast) = 0,\]
one can look at the site potential
\[W_{\rm atom}(A) = \frac{1}{2} \sum_{\rho \in \mathcal{R}, \lvert \rho \rvert = 1} V_1(\lvert A_\rho \rvert) +  \frac{1}{2} \sum_{\rho \in \mathcal{R}, \lvert \rho \rvert = \sqrt{2}} V_{2}(\lvert A_\rho).\]
We can now set $K_1 = V_1''(r^\ast)$, $K_2 = V_2''(r^\ast)$, $a_1 =r^\ast -\frac{V_1'(r^\ast)}{V_1''(r^\ast)}$, and $a_2 =\sqrt{2} r^\ast -\frac{V_2'(\sqrt{2} r^\ast)}{V_2''(\sqrt{2} r^\ast)}$. As long as $K_1, K_2, a_1, a_2 > 0$, the above analysis applies directly since the linearization $K$ is the same.

\section{Solving the Continuous Equations}\label{section:Solvingthecontinuousequations}

\subsection{The Linearized System}

Let us first recall standard results for the linear(-ized) system.

\begin{thm} \label{thm:existsunique}
Let $\Omega \subset \R^d$ be an open, bounded set, let $L \in \R^{d \times d \times d \times d}$, $f \in L^2(\Omega;\R^d)$, $F \in L^2(\Omega; \R^{d \times d})$ and $g \in H^1(\Omega;\R^d)$. 
Furthermore, assume $\lambda_{\rm LH}(L)>0$. Then there is one and only one weak solution $u \in g + H^1_0(\Omega;\R^d)$ of
\[-\divo (L[\nabla u]) = f - \divo F. \]
\end{thm}

\begin{thm} \label{thm:optimalregularity}
Let $m \in \N_0$, let $\Omega\subset \R^d$ be an open, bounded set with $C^{m+2}$ boundary, let $L \in C^{m,1}(\Omega;\R^{d \times d \times d \times d})$, $f \in H^m(\Omega;\R^d)$, $F \in H^{m+1}(\Omega; \R^{d \times d})$ and $g \in H^{m+2}(\Omega;\R^d)$. 
Furthermore, assume $\lambda_{\rm LH}(L(x)) \geq \lambda_0$ for some $\lambda_0 > 0$ and all $x\in \Omega$.
Assume that $u \in g + H^1_0(\Omega;\R^d)$ is a weak solution of
\[ -\divo (L[\nabla u]) = f - \divo F. \]
Then $u \in H^{m+2}(\Omega;\R^d)$ and there is a $c=c(m,\Omega,\lVert L \rVert_{C^{m,1}}, \lambda)>0$, such that
\[ \lVert \nabla^{m+2} u \rVert_{L^2} \leq c (\lVert f \rVert_{H^m} + \lVert \nabla F \rVert_{H^m} + \lVert g \rVert_{H^{m+2}}). \]
\end{thm}

We only need this theorem for constant $L$. Reformulating these results we get:
\begin{cor} \label{cor:linearoperator}
Let $m \in \N_0$, let $\Omega \subset \R^d$ be an open, bounded set with $C^{m+2}$ boundary, let $L \in \R^{d \times d \times d \times d}$ and assume $\lambda_{\rm LH}(L)>0$. Then the mapping
\[  u \mapsto \divo ( L[\nabla u] )\]
is a linear isomorphism from $H^{m+2}(\Omega;\R^d)\cap H^1_0(\Omega;\R^d) $ onto $H^m(\Omega;\R^d)$.
\end{cor}
\begin{proof}
These statements are rather standard and can be found in the literature. See, e.g., \cite[Corollary 3.46]{giaquintamartinazzi} and \cite[Theorem 4.14]{giaquintamartinazzi}.
\end{proof}
\subsection{Local Solutions of the Nonlinear Problem}

We now improve the linearized result to a local result for the nonlinear problem with the help of an implicit function theorem.

\begin{thm} \label{thm:contresult}
Let $m \in \N_0$, $d<2m+2$ and let $\Omega \subset \R^d$ be an open, bounded set with $C^{m+2}$-boundary. Let $r_0>0$, $W_{\rm atom} \in C^{m+3}(\overline{B_{r_0}((A_0 \rho)_{\rho \in \mathcal{R}})})$ and assume that $\lambda_{\rm LH}(A_0)>0$. Then there are constants $\kappa_1,\kappa_2>0$ such that for all $g\in H^{m+2}(\Omega;\R^d)$ and $f \in H^m(\Omega;\R^d)$, that satisfy $\lVert g-y_{A_0} \rVert_{H^{m+2}(\Omega;\R^d)} < \kappa_1$ and $\lVert f \rVert_{H^m(\Omega;\R^d)} < \kappa_1$, the problem
\begin{align*}
-\divo( DW_{\rm CB}(\nabla y(x))) &= f(x), \quad \text{ if } x \in \Omega, \\
	y(x) &= g(x),  \quad \text{ if } x \in \partial\Omega,
\end{align*}
has exactly one weak solution with $\lVert y-g \rVert_{H^{m+2}(\Omega;\R^d)} < \kappa_2$. Furthermore, we always have that
\[ \sup_x \lvert ((\nabla y(x)-A_0)\rho)_{\rho \in \mathcal{R}} \rvert <r_0,\]
that $y$ is a $W^{1,\infty}$-local minimizer of $E(\cdot ;f)$ restricted to $y=g$ on $\partial \Omega$ and that $y$ depends $C^1$ on $f$ and $g$ in the norms used above.
\end{thm}
Let us start with an important statement on compositions:
\begin{lem} \label{lem:composition}
Let $m \in \N_0$, $d<2m+2$ and let $\Omega \subset \R^d$ be an open, bounded set with Lipschitz boundary. Let $V \subset \R^{d \times \mathcal{R}}$ and $W_{\rm atom} \in C_b^{m+3}(V)$ with uniform continuous highest derivatives.

Define the operator $F\colon B \mapsto DW_{\rm CB} \circ B$. We claim that
\[\{B\in H^{m+1}(\Omega;\R^{d \times d}) \colon \inf\limits_{x \in \Omega} \dist ( (B(x) \rho)_{\rho \in \mathcal{R}}, V^c)>0\}\]
is open in $H^{m+1}(\Omega;\R^{d \times d})$ and
\[F\colon \{B\in H^{m+1}(\Omega;\R^{d \times d}) \colon \inf\limits_{x \in \Omega} \dist ( (B(x) \rho)_{\rho \in \mathcal{R}}, V^c)>0\} \to H^{m+1}(\Omega;\R^{d \times d})\]
is well-defined and $C^1$ with
\[DF(B)[H](x) = D^2 W_{\rm CB} (B(x))[H(x)].\]
\end{lem}
\begin{proof}
This is contained in \cite[II.~Thm.4.1]{valentbvps}.
\end{proof}

\begin{proof}[Proof of Theorem \ref{thm:contresult}]
Let $X= H^{m+2}(\Omega;\R^d)$, $X_0= H^{m+2}(\Omega;\R^d) \cap H^1_0(\Omega;\R^d)$ and $Y = H^m(\Omega;\R^d)$. Define $T \colon B_{r_1}^{X_0}(0) \times B_{r_2}^{X}(0) \times Y \to Y$,
\[ T(u,h,f) = -\divo(DW_{\rm CB}(A_0 + \nabla h(x) + \nabla u(x))) - f(x).\]
If we choose $r_1,r_2>0$ small enough, then we always have
\[ \sup\limits_{x \in \Omega} \lvert ((\nabla h (x) + \nabla u (x))\rho)_{\rho \in \mathcal{R}}\rvert < r_0, \]
since $H^{m+2}(\Omega; \R^d) \hookrightarrow  C^1_b(\Omega;\R^d)$. Using the properties of $F$ from Lemma \ref{lem:composition} with $V= B_{r_0}((A_0 \rho)_{\rho \in \mathcal{R}})$, this implies that $T$ is well-defined, is in $C^1$ and
\[\partial_u T(u,h,f)[v](x) = -\divo ( D^2 W_{\rm CB}(A_0+ \nabla u(x) + \nabla h(x))[\nabla v(x)]).\]
In particular,
\[\partial_u T(0,0,0)[v](x) = -\divo ( D^2 W_{\rm CB}(A_0)[\nabla v(x)]).\]
Since $D^2 W_{\rm CB}(A_0)$ satisfies the Legendre-Hadamard condition, the invertibility of $\partial_u T(0,0,0) \colon X_0 \to Y$ follows from Corollary \ref{cor:linearoperator}.
Now the main statement on existence, uniqueness and $C^1$-dependence follows from a standard Banach space implicit function theorem, as can be found, e.g., in \cite[Thm.~15.1 and Cor.~15.1]{deimlingnfa}, and then setting $g = y_{A_0} + h$, $y = y_{A_0} + h + u$.

Furthermore, if we choose $r_1,r_2$ even smaller, the above statements are still true and we can achieve that
\[ \sup\limits_{x \in \Omega} \lvert \nabla h (x) + \nabla u (x)\rvert < \tilde{r} \]
for all $(u,h) \in  B_{r_1}^{X_0}(0) \times B_{r_2}^{X}(0)$, where $\tilde{r}$ is such that
\[\int\limits_\Omega D^2 W_{\rm CB}(\nabla z(x))[\nabla s,\nabla s] \,dx \geq \frac{\lambda_{\rm LH}(A_0)}{2}\int\limits_\Omega \lvert \nabla s(x) \rvert^2 \,dx\]
holds for all $s\in H^1_0(\Omega;\R^d)$ and $z \in W^{1,\infty}(\Omega;\R^d)$ with $\lvert \nabla z(x) - A_0 \rvert \leq \tilde{r}$ a.e.. This is possible since $D^2 W_{\rm CB}$ is uniformly continuous.

Now, if $w\in W^{1,\infty}(\Omega;\R^d)$ is in this space close enough to $y$ and also has boundary values $g$, then
\[\lvert \nabla w(x) - A_0\rvert \leq \tilde{r},\]
and we have
\begin{align*}
	E(w;f) &= E(y;f)\\
	&+ \int\limits_0^1 \int\limits_\Omega (1-t)D^2 W_{\rm CB}( (1-t)\nabla y(x) + t \nabla w(x) )[\nabla w(x) - \nabla y(x)]^2 \,dx \,dt \\
	&\geq E(y;f) + \frac{\lambda}{2}\int\limits_\Omega \lvert \nabla w(x) - \nabla y(x)\rvert^2 \,dx.
\end{align*}
Hence, $y$ is a $W^{1,\infty}$-local minimizer of $E(\cdot;f)$ restricted to having boundary values $g$ (strongly in the $H^1_0(\Omega;\R^d)$-Norm).
\end{proof}

\section{Existence and Convergence of Solutions of the Atomistic Equations}\label{section:ConvergenceToAtomistic}

\subsection{Statement of the Main Theorem}
Let us define the following discrete norms and semi-norms:
\[ \lVert u \rVert_{\ell_\eps^2(\Lambda)} = \Big( \eps^d \sum_{x \in \Lambda} \lvert u(x) \rvert^2 \Big)^{\frac{1}{2}} \]
for any finite set $\Lambda$ and $u \colon \Lambda \to \R^d$,
\[ \lVert u \rVert_{h^1_\eps(\sinto_\eps \Omega)} = \Big( \eps^d \sum_{x \in \sinto_\eps \Omega} \lvert D_{\mathcal{R},\eps} u(x) \rvert^2 \Big)^{\frac{1}{2}}  \]
for $u \colon \Omega \cap \eps \Z^d \to \R^d$ and
\[ \lVert u \rVert_{h_\eps^{-1}(\into_\eps \Omega)} = \sup \Big\{ \eps^d \sum_{x\in\into_\eps \Omega} u(x) \varphi(x) \colon \varphi \in \mathcal{A}_\eps(\Omega,0) \text{ with }\lVert \varphi \rVert_{h^1_\eps(\sinto_\eps \Omega)}=1\Big\} \]
for $u \colon \into_\eps \Omega \to \R^d$.
The $h^1_\eps$-semi-norm is given by the semi-definite symmetric bilinear form
\[(u,v)_{h^1_\eps(\sinto_\eps \Omega)} = \eps^d\sum_{x \in \sinto_\eps \Omega} D_{\mathcal{R},\eps}u(x) \colon D_{\mathcal{R},\eps}v(x),\]
where $A \colon B = \sum_\rho \sum_j A_{j \rho} B_{j \rho}$. On $\mathcal{A}_\eps (\Omega,0)$ this is a norm and a scalar product, respectively.

Given $g \colon \partial_\eps \Omega \to \R^d$, $y \colon \Omega \cap \eps \Z^d$ minimizes $\lVert y \rVert_{h^1_\eps(\sinto_\eps \Omega)}$ under the constraint $y(x)=g(x)$ for all $x \in \partial_\eps \Omega$ if and only if $(y,u)_{h^1_\eps(\sinto_\eps \Omega)}=0$ for all $u \in \mathcal{A}_\eps(\Omega,0)$ and $y(x)=g(x)$ for all $x \in \partial_\eps \Omega$. Thus, for every $g \colon \partial_\eps \Omega \to \R^d$ there is precisely one such $y$, it depends linearly on $g$ and is the unique solution to $\divo_{\mathcal{R},\eps} D_{\mathcal{R},\eps} y = 0$ with boundary values $g$. We write $y=T_\eps g$. Accordingly, we define the semi-norm
\[ \lVert g \rVert_{\partial_\eps \Omega} = \lVert T_\eps g \rVert_{h^1_\eps(\sinto_\eps \Omega)}.\]

Given $\eps \in (0,1]$ and $f \in L^2(\Omega)$ we will write
\[\tilde{f}(x) = \fint_{Q_\eps(x)} f(z)\,dz\]
for $x \in \into_\eps \Omega$. If $\Omega$ has Lipschitz boundary and we have a deformation $y \in H^1(\Omega;\R^d)$ we will write
\[S_{\eps}y(x) = \eta_\eps \ast (y_{A_0} + E(y-y_{A_0}) ) (x)\]
for $x \in \eps \Z^d$, where $\eta_\eps$ is the standard scaled smoothing kernel and $E$ is an extension operator for all Sobolev spaces, see \cite[Chapter VI]{steinsingint}, such that every $Eu$ has support in a fixed ball $B_{R_E}(0)$. In the following $\tilde{f}$ and $S_{\eps}y$ are our reference points for the atomistic body forces, boundary conditions and deformations.

\begin{thm} \label{thm:atomisticstatictheorem}
Let $d \in \{1,2,3,4\}$ and let $\Omega \subset \R^d$ be an open, bounded set with $C^4$-boundary. Let $r_0 >0$, $W_{\rm atom} \in C^5(\overline{B_{r_0}((A_0 \rho)_{\rho \in \mathcal{R}})})$ and assume $\lambda_{\rm atom}(A_0)>0$. Then there are constants $K_1, K_2, K_3>0$ such that for every $f \in H^2(\Omega;\R^d)$ with $\lVert f \rVert_{H^2(\Omega;\R^d)} \leq K_1$, $g \in H^4(\Omega;\R^d)$ with $\lVert g - y_{A_0} \rVert_{H^4(\Omega;\R^d)} \leq K_1$, $\eps \in (0,1]$, $\gamma \in [\frac{d}{2},2]$, $f_{\rm atom} \colon \into_\eps \Omega \to \R^d$ with $\lVert f_{\rm atom} - \tilde{f} \rVert_{h_\eps^{-1}(\into_\eps \Omega)} \leq K_2 \eps^\gamma$, and $g_{\rm atom} \colon \partial_\eps \Omega \to \R^d$ with $\lVert g_{\rm atom} - S_{\eps}y \rVert_{\partial_\eps \Omega} \leq K_2 \eps^\gamma$, where $y$ is the continuous solution corresponding to $f$ and $g$ given by Theorem \ref{thm:contresult}, there is a unique $y_{\rm atom} \in \mathcal{A}_\eps(\Omega, g_{\rm atom})$ with $\lVert y_{\rm atom} - S_{\eps}y \rVert_{h^1_\eps(\sinto_\eps \Omega)} \leq K_3 \eps^\gamma$ such that
\[ -\divo_{\mathcal{R},\eps} \big( DW_{\rm atom}(D_{\mathcal{R},\eps} y_{\rm atom} (x))\big) = f_{\rm atom}(x)\]
for all $x \in \into_\eps\Omega$. Furthermore, $y_{\rm atom}$ is a strict local minimizer of $E_\eps(\cdot,f_{\rm atom},g_{\rm atom})$.

Additionally, there is a $K_4>0$ such that whenever $\gamma \in (1,2]$ and $E(y-y_{A_0}) \in C^{2,(\gamma-1)}(\Omega)$ then
\[\lVert y_{\rm atom} - y \rVert_{h^1_\eps(\into_\eps \Omega)} \leq (K_3 + K_4 \lVert \nabla^2 E(y-y_{A_0}) \rVert_{C^{0,\gamma-1}}) \eps^\gamma.\]
\end{thm}

\begin{rem}
If $d=3$ and $\gamma = \frac{3}{2}$ the assumption in the additional statement is automatically satisfied since $E(y-y_{A_0})\in H^4(B_{R_E}(0)) \hookrightarrow C^{2, \frac{1}{2}}(B_{R_E}(0))$.
\end{rem}

\subsection{A Quantitative Implicit Function Theorem}
The proof of Theorem \ref{thm:atomisticstatictheorem} relies on an implicit function theorem, which will eventually yield the desired solution to the atomistic equations if an approximate solution can be found with good estimates on the residuum, as well as invertibility, boundedness, and continuity of certain partial derivatives. The approximate solution in our case will be a smooth approximation of the solution to the corresponding continuous equations with the Cauchy-Born energy density. In order to obtain the strong estimates on the rate of convergence stated in Theorem \ref{thm:atomisticstatictheorem}, we will formulate a quantitative implicit function theorem which also allows for a small parameter.

\begin{thm} \label{thm:quantitativeimplicitfunctiontheorem}
Let $X$ be a Banach space and $Y, Z$ normed spaces, $U \subset X$, $V \subset Y$ open and $F \colon U \times V \to Z$ Fréchet-differentiable. Assume that $\partial_u F (0,0) \colon X \to Z$ is invertible. Furthermore, assume that there are $\rho,\tau,\kappa_1,\kappa_2,\kappa_3>0$ and a function $\omega \colon [0,\infty)^2 \to [0,\infty]$, non-decreasing in both variables, such that $\overline{B_\rho (0)} \subset U$, $\overline{B_\tau (0)} \subset V$,
\begin{align*}
\lVert F(0,0) \rVert_Z &\leq \kappa_1,\\
\lVert \partial_u F(0,0)^{-1} \rVert_{L(Z,X)} &\leq \kappa_2,\\
\lVert \partial_h F(u,h) \rVert_{L(Y,Z)} &\leq \kappa_3 \quad \forall (u,h)\in \overline{B_\rho (0)} \times \overline{B_\tau(0)} ,\\
\lVert \partial_u F(0,0) - \partial_u F(u,h) \rVert_{L(X,Z)} &\leq \omega(\lVert u \rVert_X,\lVert h \rVert_Y) \quad \forall (u,h)\in \overline{B_\rho (0)} \times \overline{B_\tau(0)} ,\\
\kappa_2 \omega(\rho,\tau) &< 1, \text{ and}\\
\kappa_2 \big(\kappa_1 + \kappa_3 \tau + \int_0^\rho \omega \big(t,\frac{\tau}{\rho} t \big) \,dt \big) &\leq \rho.
\end{align*}
Then, for every $h \in \overline{B_\tau(0)}$ there is a unique $u \in \overline{B_\rho(0)}$ with $F(u,h)=0$.
\end{thm}
\begin{proof}
Let
\[G_h(u) = u - \partial_u F(0,0)^{-1} F(u,h).\]
If $\lVert u \rVert \leq \rho$ and $\lVert h \rVert \leq \tau$ then
\[G_h(u)= \partial_u F(0,0)^{-1} \big(- F(0,0) + \partial_u F(0,0)u + F(0,0) -F(u,h) \big).\]
Therefore,
\begin{align*}
\lVert G_h(u) \rVert &\leq \kappa_2 \kappa_1 + \kappa_2 \Big\lVert \int_0^1 (\partial_u F(0,0)-\partial_u F(tu,th))u - \partial_h F(tu,th)h \,dt \Big\rVert\\
&\leq \kappa_2 \big(\kappa_1 + \int_0^\rho \omega \big(t,\frac{\tau}{\rho} t \big) \,dt + \kappa_3 \tau \big) \leq \rho.
\end{align*}
Furthermore, for $u,v \in \overline{B_\rho (0)}$ we have
\begin{align*}
\lVert G_h(u)-G_h(v) \rVert &\leq \kappa_2 \Big\lVert \int_0^1 (\partial_u F(u+t(v-u),h) -\partial_u F(0,0))(v-u) \,dt \Big\rVert \\
&\leq \kappa_2 \omega_1 (\rho,\tau) \lVert v-u \rVert \\
&\leq \frac{1}{2} \lVert v-u \rVert.
\end{align*}
Hence, $G_h$ has a unique fixed point in $u \in \overline{B_\rho (0)}$.
\end{proof}
More precisely, we want to use the following more specific corollary:
\begin{cor} \label{cor:quantitativeimplicitfunctiontheorem}
Let $d \in \{1,2,3,4\}$. Assume we have a family $F_\eps \colon U_\eps \times V_\eps \to Z_\eps$ with $\eps \in (0,1]$ and $U_\eps \subset X_\eps, V_\eps \subset Y_\eps$ open, where $Y_\eps, Z_\eps$ are normed spaces and $X_\eps$ Banach spaces. Furthermore, assume that the $F_\eps$ are Fréchet-differentiable and we have fixed $r_1,r_2>0$ such that $\overline{B_{r_1\eps^{\frac{d}{2}}}(0)} \subset U_\eps$ and $\overline{B_{r_2\eps^{\frac{d}{2}}}(0)} \subset V_\eps$.
Now, assume there are $A,M_1,M_2,M_3,M_4>0$, such that
\begin{align*}
\lVert F_\eps(0,0) \rVert_{Z_\eps} &\leq A \eps^2,\\
\lVert \partial_u F_\eps(0,0)^{-1} \rVert_{L(Z_\eps,X_\eps)} &\leq M_1,\\
\lVert \partial_h F_\eps(u,h) \rVert_{L(Y_\eps,Z_\eps)} &\leq M_2 \qquad \forall (u,h)\in \overline{B_{r_1\eps^{\frac{d}{2}}}(0)} \times \overline{B_{r_2\eps^{\frac{d}{2}}}(0)} ,\\
\lVert \partial_u F_\eps(0,0) - \partial_u F_\eps(u,h) \rVert_{L(X_\eps,Z_\eps)} &\leq M_3 \eps^{-\frac{d}{2}} (\lVert u \rVert_{X_\eps}+\lVert h \rVert_{Y_\eps}) \\
&\qquad \qquad \forall (u,h)\in \overline{B_{r_1\eps^{\frac{d}{2}}}(0)} \times \overline{B_{r_2\eps^{\frac{d}{2}}}(0)}
\end{align*}
and
\[A \leq \min\Big\{ \frac{r_1}{3M_1}, \frac{1}{9 M_1^2 M_3}\Big\}.\]
If we now set $\rho_\eps = \lambda_1 \eps^\gamma$ and $\tau_\eps = \lambda_2 \eps^\gamma$ for arbitrary $\gamma \in \big[\frac{d}{2},2\big]$ and
\begin{align*}
	\lambda_1 &= \min\Big\{r_1,\frac{1}{3 M_1 M_3}\Big\},\\
	\lambda_2 &= \min\Big\{r_2,\frac{1}{3 M_1 M_3}, \frac{\lambda_1}{3 M_1 M_2}\Big\},
\end{align*}
then for every $\eps \in (0,1]$ and every $h_\eps \in \overline{B_{\tau_\eps}(0)}$ there is a unique $u_\eps \in \overline{B_{\rho_\eps}(0)}$ with $F_\eps(u_\eps,h_\eps)=0$.
\end{cor}
\begin{proof}
We set 
\begin{align*}
\kappa_1&= A \eps^2,\\
\kappa_2&= M_1,\\
\kappa_3&= M_2, \text{ and}\\
\omega(s,t) &= M_3 \eps^{-\frac{d}{2}} (s+t).
\end{align*}
A simple calculation gives
\[\kappa_2 \omega(\rho_\eps,\lambda_\eps) = 2 M_1 M_3 \eps^{\gamma-\frac{d}{2}} (\lambda_1 + \lambda_2) \leq \frac{2}{3} < 1 \]
and
\begin{align*}
\kappa_2 \big(\kappa_1 &+ \kappa_3 \tau + \int_0^\rho \omega \big(t,\frac{\tau}{\rho} t \big) \,dt \big)\\
 &\leq  A M_1 \eps^2 + M_1 M_2 \lambda_2 \eps^\gamma + M_1 M_3 \eps^{-\frac{d}{2}} \big(1+ \frac{\lambda_2}{\lambda_1}\big)\frac{1}{2} \lambda_1^2 \eps^{2 \gamma}\\
&\leq \frac{\lambda_1 \eps^2}{3} + \frac{\lambda_1 \eps^\gamma}{3} + \frac{\lambda_1 \eps^{2 \gamma- \frac{d}{2}}}{3}\\
&\leq \rho.
\end{align*}
We can therefore apply Theorem \ref{thm:quantitativeimplicitfunctiontheorem}.
\end{proof}
\begin{rem}
Without the smallness condition on $A$ the theorem is still true for all $\eps$ small enough if $d \leq 3$. In the case $\gamma = 2$ this requires a different choice of $\lambda_1$ and $\lambda_2$, e.g. $\lambda_1 = 3 M_1 A$ and $\lambda_2 = \frac{A}{M_2}$.
\end{rem}
 
\begin{rem}
In contrast to previous work, e.g. \cite{ortnertheil13}, we have a much larger space of data and cannot just map body forces to solutions with a quantitative inverse function theorem.
\end{rem}

\subsection{Residual Estimates}

\begin{prop} \label{prop:ell2residuum}
Let $V \subset \R^{d \times \mathcal{R}}$ be open and $W_{\rm atom} \in C^4_b(V)$. Let $f \in L^2(\Omega;\R^d)$ and set as before
\[ \tilde{f}(x) = \fint_{Q_\eps(x)} f(a)\,da\]
for $x \in \into_\eps \Omega$. Furthermore let $\eps \in (0,1]$ and $y \in C^{3,1}(\R^d;\R^d)$ with
\begin{align*}
\co \{D_{\mathcal{R},\eps} y (\hat{x}+ \eps \sigma), (\nabla y (x) \rho)_{\rho \in \mathcal{R}}\} \subset V
\end{align*}
for all $x \in \Omega_\eps$ and $\sigma \in \mathcal{R} \cup \{0\}$. Then we have
\begin{align*}
\big\lVert -\tilde{f} &- \divo_{\mathcal{R},\eps} \big( DW_{\rm atom}(D_{\mathcal{R},\eps} y)\big) \big\rVert_{\ell_\eps^2(\into_\eps \Omega)}\\
&\leq \lVert -f - \divo DW_{\rm CB}(\nabla y)\rVert_{L^2(\Omega_\eps; \R^d)}+ C \eps^2 \Big\lVert \lVert \nabla^4 y \rVert_{L^\infty(B_{\eps R}(x))}\\
&+ \lVert \nabla^3 y \rVert_{L^\infty(B_{\eps R}(x))}^\frac{3}{2} + \lVert \nabla^2 y \rVert_{L^\infty(B_{\eps R}(x))}^3 + \eps\lVert \nabla^3 y \rVert_{L^\infty(B_{\eps R}(x))}^2 \Big\rVert_{L^2(\Omega_\eps)},
\end{align*}
where $\Omega_\eps = \bigcup_{z \in \into_\eps \Omega} Q_\eps(z)$, $R=2R_{\rm max}+\frac{3\sqrt{d}}{2}$ and $C = C(d,\mathcal{R}, \lVert D^2 W_{\rm atom} \rVert_{C^2(V)}) >0$.
\end{prop}
\begin{proof}
For $x\in \Omega_\eps$, $\sigma \in \mathcal{R}$ and $\rho\in \mathcal{R} \cup \{0\}$ set
\begin{align*}
r_{1,\eps}(x;\sigma,\rho) &= \frac{y(\hat{x} - \eps \rho +\eps \sigma) - y(\hat{x} - \eps \rho)}{\eps} - \nabla y(x)\sigma,\\
r_{2,\eps}(x;\sigma,\rho) &= \frac{y(\hat{x} - \eps \rho +\eps \sigma) - y(\hat{x} - \eps \rho)}{\eps} - \nabla y(x)\sigma\\
&- \frac{1}{2}\eps \nabla^2 y(x)[\sigma - \rho +\frac{\hat{x} - x}{\eps}, \sigma - \rho+\frac{\hat{x} - x}{\eps}]\\
&+ \frac{1}{2}\eps \nabla^2 y(x)[-\rho+\frac{\hat{x} - x}{\eps},-\rho+\frac{\hat{x} - x}{\eps}],\\
r_{3,\eps}(x;\sigma,\rho) &= \frac{y(\hat{x} - \eps \rho +\eps \sigma) - y(\hat{x} - \eps \rho)}{\eps} - \nabla y(x)\sigma\\
&- \frac{1}{2}\eps \nabla^2 y(x)[\sigma - \rho +\frac{\hat{x} - x}{\eps}, \sigma - \rho+\frac{\hat{x} - x}{\eps}]\\
&+ \frac{1}{2}\eps \nabla^2 y(x)[-\rho+\frac{\hat{x} - x}{\eps},-\rho+\frac{\hat{x} - x}{\eps}]\\
&- \frac{1}{6}\eps^2 \nabla^3 y(x)[\sigma - \rho+\frac{\hat{x} - x}{\eps}, \sigma - \rho+\frac{\hat{x} - x}{\eps}, \sigma - \rho+\frac{\hat{x} - x}{\eps}]\\
&+ \frac{1}{6}\eps^2 \nabla^3 y(x)[-\rho+\frac{\hat{x} - x}{\eps},-\rho+\frac{\hat{x} - x}{\eps},-\rho+\frac{\hat{x} - x}{\eps}].
\end{align*}
First order Taylor expansions with integral remainder of $y(\hat{x} - \eps \rho +\eps \sigma)$ and $y(\hat{x} - \eps \rho)$ at $x$ give the estimate
\[\lvert r_{1,\eps}(x;\sigma,\rho) \rvert \leq \eps \bar{R}^2 \lVert \nabla^2 y \rVert_{L^\infty(B_{\eps \bar{R}}(x))},\]
where $\bar{R} = 2 R_{max} + \frac{1}{2}\sqrt{d}$ and we have used that $\lvert \hat{x}-x \rvert \leq \frac{1}{2} \eps \sqrt{d}$. Similarly, second and third order Taylor expansions give
\begin{align*}
\lvert r_{2,\eps}(x;\sigma,\rho) \rvert &\leq \frac{1}{3} \eps^2 \bar{R}^3 \lVert \nabla^3 y \rVert_{L^\infty(B_{\eps \bar{R}}(x))}, \\
\lvert r_{3,\eps}(x;\sigma,\rho) \rvert &\leq \frac{1}{12} \eps^3 \bar{R}^4 \lVert \nabla^4 y \rVert_{L^\infty(B_{\eps \bar{R}}(x))}.
\end{align*}

Now, doing a second order Taylor expansion of $DW_{\rm atom}$ at $(\nabla y (x) \rho)_{\rho \in \mathcal{R}}$ with integral remainder, using the definition of $r_{3,\eps}$ in the first order term, the definition of $r_{2,\eps}$ in the second order term and the definition of $r_{1,\eps}$ in the remainder and then collecting the terms with the same exponent in $\eps$ gives
\begin{align*}
-f(x)-\divo_{\mathcal{R},\eps} \big( DW_{\rm atom}(D_{\mathcal{R},\eps} y (\hat{x})) \big)=\eps^{-1} I_{-1} + \eps^0 I_0 + \eps^1 I_1 +R_\eps(x),
\end{align*}
where
\[I_{-1} = -\sum\limits_{\rho \in \mathcal{R}} D_{e_\rho} W_{\rm atom}((\nabla y (x) \rho)_{\rho \in \mathcal{R}}) - D_{e_\rho} W_{\rm atom}((\nabla y (x) \rho)_{\rho \in \mathcal{R}}) = 0\]
and
\begin{align*}
(I_0)_j &=  -f_j(x) - \sum\limits_{\rho,\sigma \in \mathcal{R}} D^2 W_{\rm atom}((\nabla y (x) \rho)_{\rho \in \mathcal{R}})\Big[e_j \otimes e_\rho,\frac{1}{2} \nabla^2 y(x) \\
&\quad\Big(\big[\sigma + \frac{\hat{x}-x}{\eps}\big]^2 - \big[\frac{\hat{x}-x}{\eps}\big]^2 -\big[\sigma - \rho + \frac{\hat{x}-x}{\eps}\big]^2 + \big[-\rho + \frac{\hat{x}-x}{\eps}\big]^2 \Big) \otimes e_\sigma\Big]  \displaybreak[0]\\
&=  -f_j(x) - \sum\limits_{\rho,\sigma \in \mathcal{R}} D^2 W_{\rm atom}((\nabla y (x) \rho)_{\rho \in \mathcal{R}})\Big[e_j \otimes e_\rho, \nabla^2 y(x)[\sigma,\rho] \otimes e_\sigma\Big]  \displaybreak[0]\\
&= -f_j(x) -  \sum\limits_{i} \sum\limits_{\rho,\sigma \in \mathcal{R}} D^2 W_{\rm atom}((\nabla y (x) \rho)_{\rho \in \mathcal{R}})\Big[((e_j \otimes e_i)\rho) \otimes e_\rho,\\
&\quad \nabla^2 y(x)[\sigma,e_i] \otimes e_\sigma\Big]\displaybreak[0]\\
&= -f_j(x) -  \sum\limits_{i} \frac{\partial}{\partial x_i}\sum\limits_{\rho \in \mathcal{R}} DW_{\rm atom}((\nabla y (x) \rho)_{\rho \in \mathcal{R}})\Big[((e_j \otimes e_i)\rho) \otimes e_\rho\Big]
\displaybreak[0]\\
&= -f_j(x) -  \sum\limits_{i} \frac{\partial}{\partial x_i} DW_{\rm CB}(\nabla y (x)) [e_j \otimes e_i]
\displaybreak[0]\\
&= \big(-f(x) - \divo DW_{\rm CB}(\nabla y(x))\big)_j 
\end{align*}
and
\begin{align*}
(I_1)_j =&  -\sum\limits_{\rho,\sigma \in \mathcal{R}} D^2 W_{\rm atom}((\nabla y (x) \rho)_{\rho \in \mathcal{R}})\Big[e_j \otimes e_\rho,\frac{1}{6}\nabla^3 y(x)\\
&\quad \Big( \big[\sigma + \frac{\hat{x}-x}{\eps}\big]^3-\big[\frac{\hat{x}-x}{\eps}\big]^3-\big[\sigma-\rho + \frac{\hat{x}-x}{\eps}\big]^3+\big[-\rho + \frac{\hat{x}-x}{\eps}\big]^3 \Big)\otimes e_\sigma\Big] \displaybreak[0]\\
&-\frac{1}{2}\sum\limits_{\rho,\sigma,\tau \in \mathcal{R}} D^3 W_{\rm atom}((\nabla y (x) \rho)_{\rho \in \mathcal{R}})\big[e_j \otimes e_\rho \big]\\
&\quad\bigg(\Big[\frac{1}{2}\nabla^2 y(x)\Big(\big[\sigma+ \frac{\hat{x}-x}{\eps}\big]^2 - \big[\frac{\hat{x}-x}{\eps}\big]^2 \Big) \otimes e_\sigma,\\
&\quad \frac{1}{2}\nabla^2 y(x)\Big(\big[\tau+ \frac{\hat{x}-x}{\eps}\big]^2 - \big[\frac{\hat{x}-x}{\eps}\big]^2 \Big) \otimes e_\tau \Big]\\
&\quad-\Big[\frac{1}{2}\nabla^2 y(x)\Big(\big[\sigma-\rho+ \frac{\hat{x}-x}{\eps}\big]^2 - \big[\frac{\hat{x}-x}{\eps}-\rho\big]^2 \Big) \otimes e_\sigma,\\
&\quad \frac{1}{2}\nabla^2 y(x)\Big(\big[\tau-\rho+ \frac{\hat{x}-x}{\eps}\big]^2 - \big[\frac{\hat{x}-x}{\eps}-\rho\big]^2 \Big) \otimes e_\tau \Big]\bigg)\displaybreak[0]\\
=& -\sum\limits_{\rho,\sigma \in \mathcal{R}} D^2 W_{\rm atom}((\nabla y (x) \rho)_{\rho \in \mathcal{R}})\Big[e_j \otimes e_\rho, \frac{1}{2}\nabla^3 y(x)[\sigma,\rho,\sigma- \rho + 2\frac{\hat{x}-x}{\eps}]\otimes e_\sigma\Big]\\
&- \frac{1}{2}\sum\limits_{\rho,\sigma,\tau \in \mathcal{R}} D^3 W_{\rm atom}((\nabla y (x) \rho)_{\rho \in \mathcal{R}})\big[e_j \otimes e_\rho \big]\\
&\quad\bigg(\Big[\nabla^2 y(x)[\sigma,\rho] \otimes e_\sigma, \nabla^2 y (x)\big[\tau, \frac{1}{2}\tau + \frac{\hat{x}-x}{\eps}\big] \otimes e_\tau \Big]\\
&\quad+\Big[\nabla^2 y (x)\big[\sigma, \frac{1}{2}\sigma + \frac{\hat{x}-x}{\eps}\big] \otimes e_\sigma, \nabla^2y(x)[\tau,\rho] \otimes e_\tau\Big]\\
&\quad-\Big[\nabla^2 y (x)[\sigma,\rho] \otimes e_\sigma, \nabla^2 y (x)[\tau,\rho] \otimes e_\tau\Big]\bigg)\displaybreak[0]\\
=& -\sum\limits_{\rho,\sigma \in \mathcal{R}} D^2 W_{\rm atom}((\nabla y (x) \rho)_{\rho \in \mathcal{R}})\Big[e_j \otimes e_\rho, \nabla^3 y(x)[\sigma,\rho,\frac{\hat{x}-x}{\eps}]\otimes e_\sigma\Big]\\
&-\frac{1}{2}\sum\limits_{\rho,\sigma,\tau \in \mathcal{R}} D^3 W_{\rm atom}((\nabla y (x) \rho)_{\rho \in \mathcal{R}})\big[e_j \otimes e_\rho \big]\\
&\quad\bigg(\Big[\nabla^2y(x)[\sigma,\rho]\otimes e_\sigma, \nabla^2y(x)[\tau,\frac{\hat{x}-x}{\eps}]\otimes e_\tau\Big]\\
&\quad+\Big[\nabla^2y(x)[\sigma,\frac{\hat{x}-x}{\eps}]\otimes e_\sigma, \nabla^2y(x)[\tau,\rho] \otimes e_\tau\Big]\bigg),
\end{align*}
where in the last equality we applied the symmetry condition in the form of Lemma \ref{lem:symmetry}. While the last expression is not necessarily zero, it is linear in $\frac{\hat{x}-x}{\eps}$, with coefficients depending on $x$. Therefore, the average $\frac{1}{2}(I_1(x)+I_1(\bar{x}))$ is actually of higher order. Here $\bar{x}$ denotes the almost everywhere uniquely defined point in the same cube as $x$ such that $\frac{1}{2}(x+\bar{x}) = \hat{x}$. To be more precise, we have
\begin{align*}
\Big\lvert \frac{\eps}{2}(I_1(x)&+I_1(\bar{x})) \Big\rvert \displaybreak[0]\\
\leq\, & \frac{1}{2}\eps^2 \lvert\mathcal{R}\rvert^2 \lVert D^2W_{\rm atom} \rVert_\infty \sqrt{d} \lVert \nabla^4 y \rVert_{L^\infty(B_{\eps\sqrt{d}}(x))} R_{\rm max}^2 \frac{\sqrt{d}}{2}\\
&+ \frac{1}{2} \eps^2 \lvert\mathcal{R}\rvert^{\frac{5}{2}} \lVert D^3 W_{\rm atom} \rVert_\infty \sqrt{d} \lVert \nabla^2 y \rVert_{L^\infty(B_{\eps\sqrt{d}}(x))} R_{\rm max}^3 \frac{\sqrt{d}}{2} \lvert \nabla^3 y (x) \rvert\\
&+ \eps^2 \lvert\mathcal{R}\rvert^3 \lVert D^3 W_{\rm atom} \rVert_\infty \sqrt{d} \lVert \nabla^3 y \rVert_{L^\infty(B_{\eps\sqrt{d}}(x))} \lVert \nabla^2 y \rVert_{L^\infty(B_{\eps\sqrt{d}}(x))} R_{\rm max}^3 \frac{\sqrt{d}}{2} \\
&+ \eps^2 \lvert\mathcal{R}\rvert^{\frac{7}{2}} \lVert D^4 W_{\rm atom} \rVert_\infty \sqrt{d} \lVert \nabla^2 y \rVert_{L^\infty(B_{\eps\sqrt{d}}(x))} R_{\rm max}^4 \frac{\sqrt{d}}{2} \lvert \nabla^2 y (x) \rvert^2 \displaybreak[0]\\
\leq\, &  C\eps^2 \Big( \lVert \nabla^4 y \rVert_{L^\infty(B_{\eps\sqrt{d}}(x))} + \lVert \nabla^3 y \rVert_{L^\infty(B_{\eps\sqrt{d}}(x))}^\frac{3}{2}+\lVert \nabla^2 y \rVert_{L^\infty(B_{\eps\sqrt{d}}(x))}^3\Big).
\end{align*}
Using the bounds we have on $r_{1,\eps}$, $r_{2,\eps}$ and $r_{3,\eps}$, we estimate
\begin{align*}
\lvert (R_\eps)_j(x) \rvert &\leq \eps^2 \Big(\lvert \mathcal{R} \rvert^2 \lVert D^2 W_{\rm atom} \rVert_\infty \bar{R}^4 \frac{1}{6} \lVert \nabla^4 y \rVert_{L^\infty(B_{\eps \bar{R}}(x))}\\
&\quad+\lvert \mathcal{R} \rvert^3 \lVert D^3 W_{\rm atom} \rVert_\infty \frac{2}{3} \bar{R}^3 R_{\rm max}(R_{\rm max} + \frac{\sqrt{d}}{2}) \lvert \nabla^2 y(x)\rvert \lVert \nabla^3 y \rVert_{L^\infty(B_{\eps \bar{R}}(x))}\\
&\quad+ \lvert \mathcal{R} \rvert^4 \lVert D^4 W_{\rm atom} \rVert_\infty \frac{1}{3} \bar{R}^6 \lVert \nabla^2 y \rVert_{L^\infty(B_{\eps \bar{R}}(x))}^3 \Big) \\
&\quad+\eps^3 \Big( \lvert \mathcal{R} \rvert^3 \lVert D^3 W_{\rm atom} \rVert_\infty \bar{R}^6 \frac{1}{9} \lVert \nabla^3 y \rVert_{L^\infty(B_{\eps \bar{R}}(x))}^2 \Big) \\
&\leq C \eps^2 \big( \lVert \nabla^4 y \rVert_{L^\infty(B_{\eps \bar{R}}(x))} + \lVert \nabla^3 y \rVert_{L^\infty(B_{\eps \bar{R}}(x))}^\frac{3}{2}\\
&\quad + \lVert \nabla^2 y \rVert_{L^\infty(B_{\eps \bar{R}}(x))}^3 + \eps \lVert \nabla^3 y \rVert_{L^\infty(B_{\eps \bar{R}}(x))}^2 \big). 
\end{align*}
Combining these estimates and using $\bar{R}+\sqrt{d}=R$, we get
\begin{align*}
\big\lvert -\frac{f(x)+ f(\bar{x})}{2}&-\divo_{\mathcal{R},\eps} \big( DW_{\rm atom}(D_{\mathcal{R},\eps} y (\hat{x}))\big) \big\rvert\\
&\leq \Big\lvert \frac{-f(x)-\divo DW_{\rm CB}(\nabla y (x))- f(\bar{x})-\divo DW_{\rm CB}(\nabla y (\bar{x}))}{2} \Big\rvert\\
&\quad + C \eps^2 \Big( \lVert \nabla^4y \rVert_{L^\infty(B_{\eps R}(x))} + \lVert \nabla^3y \rVert_{L^\infty(B_{\eps R}(x))}^\frac{3}{2}\\
&\quad + \lVert \nabla^2y \rVert_{L^\infty(B_{\eps R}(x))}^3 + \eps \lVert \nabla^3 y \rVert_{L^\infty(B_{\eps R}(x))}^2 \Big).
\end{align*}
But,
\begin{align*}
\eps^d &\sum_{z \in \into_\eps \Omega} \big(-\tilde{f}(z) - \divo_{\mathcal{R},\eps} \big( DW_{\rm atom}(D_{\mathcal{R},\eps} y (z))\big)\big)^2\\
&\leq \sum_{z \in \into_\eps \Omega} \int_{Q_\eps(z)} \Big(-\frac{f(a)+f(\bar{a})}{2}- \divo_{\mathcal{R},\eps} \big(DW_{\rm atom}(D_{\mathcal{R},\eps} y (\hat{a})\big)\Big)^2\,da,
\end{align*}
which combined gives the desired result.
\end{proof}

These residual estimates are particularly strong if we combine them with the following two approximation results. We begin with a result that lets us convert $L^\infty$ estimates on small Balls into $L^p$ estimates. 
\begin{prop} \label{prop:approximation1}
For any $R>0$, $k,d \in \N$, $p\geq 1$, there is a $C=C(R,d,p)>0$ such that for any $U \subset \R^d$ measurable and $y \in W^{k,p}(U + B_{(R+1)\eps}(0);\R^d)$ we have
\begin{align*}
\Big\lVert \lVert \nabla^k (y \ast \eta_\eps) \rVert_{L^\infty(B_{\eps R}(\cdot))} \Big\rVert_{L^p(U)} \leq C \lVert \nabla^k y \rVert_{L^{p}(U+B_{(R+1)\eps}(0))},
\end{align*}
where $\eta_\eps$ is the standard scaled smoothing kernel.
\end{prop}
\begin{proof}
We directly calculate
\begin{align*}
\Big\lVert \lVert \nabla^k (y \ast \eta_\eps) \rVert_{L^\infty(B_{\eps R}(\cdot))} \Big\rVert_{L^p(U)}^p &\leq \int_U \esssup_{z \in B_{\eps R}(x)} \int_{\R^d} \eta_\eps(a)\lvert \nabla^k y(z+a) \rvert^{p}\,da\,dx\\
&\leq \lVert \eta \rVert_\infty \eps^{-d} \int_U \int_{B_{\eps(R+1)}(x)} \lvert \nabla^k y (a) \rvert^{p} \,da\,dx \\
&\leq C(d) (R+1)^d \int_{U+B_{\eps(R+1)}(x)} \lvert \nabla^k y (x) \rvert^{p} \,dx.
\end{align*}
\end{proof}
The second result is about estimating the nonlinearity for approximations.
\begin{prop} \label{prop:approximation2}
Let $d \in \{1,2,3,4\}$, $\Omega \subset \R^d$ open and bounded with Lipschitz boundary, $V \subset \R^{d \times \mathcal{R}}$ be open and $W_{\rm atom} \in C^5_b(V)$. Then, there is a $C>0$ such that for all $\eps \in (0,1]$ and all $y \in H^4(\Omega + B_{\eps}(0);\R^d)$
with
\begin{align*}
\inf_{x\in\Omega}\inf_{t \in [0,1]} \dist ((1-t)(\nabla y (x) \rho)_{\rho \in \mathcal{R}}+t(\nabla (y \ast \eta_\eps) (x) \rho)_{\rho \in \mathcal{R}}, V^c)>0,
\end{align*}
we have
\begin{align*}
\lVert \divo & DW_{\rm CB}(\nabla y(x)) - \divo DW_{\rm CB}(\nabla (y \ast \eta_\eps)(x)) \rVert_{L^2(\Omega)}\\
& \leq C \eps^2 \big( \lVert \nabla^2 y \rVert_{L^4(\Omega+B_{\eps}(0))} \lVert \nabla^3 y \rVert_{L^4(\Omega+B_{\eps}(0))} + \lVert \nabla^4 y \rVert_{L^2(\Omega+B_{\eps}(0))} \big)
\end{align*}
where $\eta_\eps$ is the standard scaled smoothing kernel.
\end{prop}
\begin{proof}
Now, since $\eta(z) = \eta(-z)$, we have
\begin{align*}
\nabla^k y(x) - \nabla^k (y \ast \eta_\eps)(x) &= \int_{\R^d} \eta_\eps(z) (\nabla^k y(x) - \nabla^k y(x+z))\,dz \\
&= \int_{\R^d} \eta_\eps(z) (\nabla^k y(x) + \nabla^{k+1} y(x)[z] - \nabla^k y(x+z))\,dz \\
&= -\int_{\R^d} \int_0^1 \eta_\eps(z) (1-t) \nabla^{k+2} y(x+tz)[z,z]\,dt\,dz. 
\end{align*}
But then
\begin{align*}
\int_\Omega \lvert &\nabla^k y(x) - \nabla^k (y \ast \eta_\eps)(x) \rvert^p \,dx\\
&= \int_\Omega \Big\lvert \int_{\R^d} \int_0^1 \eta_\eps(z) (1-t) \nabla^{k+2} y(x+tz)[z,z]\,dt\,dz \Big\rvert^p \,dx \\
&\leq \eps^{2p} \int_\Omega \int_{\R^d} \eta_{\eps}(z) \int_0^1 \lvert \nabla^{k+2} y (x+tz) \rvert^p\,dt\,dz\,dx\\
&\leq \eps^{2p} \lVert \nabla^{k+2} y \rVert_{L^p(\Omega+B_\eps(0))}^p.
\end{align*}
While we used strong differentiability in the proof, the inequality extends directly to $W^{k+2,p}(\Omega+B_\eps(0))$ by density. As in the proof of Lemma \ref{lem:composition} we get
\begin{align*}
DW_{\rm CB}(\nabla y) &\in H^3(\Omega; \R^{d}),\\
DW_{\rm CB}(\nabla (y\ast \eta_\eps)) &\in H^3(\Omega; \R^{d}),\\
D^2W_{\rm CB}((1-t)\nabla y + t \nabla(y \ast \eta_\eps)) &\in H^3(\Omega; \R^{d \times d}),
\end{align*}
and thus
\begin{align*}
\lVert \divo & DW_{\rm CB}(\nabla y(x)) - \divo DW_{\rm CB}(\nabla (y \ast \eta_\eps)(x)) \rVert_{L^2(\Omega)}\\
&\leq \int_0^1 \lVert \divo D^2W_{\rm CB}((1-t)\nabla y + t \nabla(y \ast \eta_\eps))[\nabla y - \nabla(y \ast \eta_\eps)] \rVert_{L^2(\Omega)} \,dt\\
&\leq C  \big(\lVert \nabla^2 y - \nabla^2 (y \ast \eta_\eps) \rVert_{L^2(\Omega)}\\
&\quad + \lVert \nabla y - \nabla (y \ast \eta_\eps) \rVert_{L^4(\Omega)} (\lVert \nabla^2 y\rVert_{L^4(\Omega)}+\lVert \nabla^2 (y \ast \eta_\eps)\rVert_{L^4(\Omega)})\big)\\
&\leq C \eps^2 \big( \lVert \nabla^4 y \rVert_{L^2(\Omega+B_{\eps}(0))}  + \lVert \nabla^3 y \rVert_{L^4(\Omega+B_{\eps}(0))} \lVert \nabla^2 y \rVert_{L^4(\Omega+B_{\eps}(0))}  \big),
\end{align*}
where we used the inequality from above with $k=p=2$ or $k=1$ and $p=4$, respectively. 
\end{proof}

\subsection{Proof of the Main Theorem}
We will need a discrete Poincaré-inequality:
\begin{prop} \label{prop:discretepoincare}
Let $\Omega \subset \R^d$ be open and bounded. Then there is a $C_{\rm P}(\Omega) >0$ such that for all $\eps \in (0,1]$ and $u \in \mathcal{A}_\eps(\Omega,0)$ we have
\[ \lVert u \rVert_{\ell_\eps^2 (\into_\eps \Omega)} \leq C_{\rm P} \lVert u \rVert_{h_\eps^1 (\sinto_\eps \Omega)} \]
and for $u \colon \into_\eps \Omega \to \R^d$ we have
\[ \lVert u \rVert_{h^{-1}_\eps (\into_\eps \Omega)} \leq C_{\rm P} \lVert u \rVert_{\ell_\eps^2 (\into_\eps \Omega)}.\]
\end{prop}
\begin{proof}
Set $M_\eps = \big\lceil \frac{\diam \Omega}{\eps} \big\rceil$, fix $\rho \in \mathcal{R}$ and extend $u$ by $0$ to all of $\eps \Z^d$. Then,
\[ u(x) = -\eps \sum\limits_{k=1}^{M_\eps} \frac{u(x+k\eps \rho) - u(x+(k-1)\eps \rho)}{\eps} \]
for all $x \in \Omega \cap \eps \Z^d$ and thus
\begin{align*}
\eps^d \sum_{x \in \into_\eps \Omega} \lvert u(x) \rvert^2 &\leq \eps^{d+2} \sum_{x \in \into_\eps \Omega} \Big(\sum\limits_{k=1}^{M_\eps} \lvert D_{\mathcal{R},\eps} u (x+ (k-1)\eps \rho) \rvert \Big)^2 \\
&\leq \eps^{d+2} M_\eps \sum_{x \in \into_\eps \Omega} \sum\limits_{k=1}^{M_\eps} \lvert D_{\mathcal{R},\eps} u (x+ (k-1)\eps \rho) \rvert^2 \\
&\leq \eps^{d+2} M_\eps^2 \sum_{x \in \sinto_\eps \Omega} \lvert D_{\mathcal{R},\eps} u (x) \rvert^2 \\
&\leq (\diam \Omega +1)^2 \eps^d \sum_{x \in \sinto_\eps \Omega} \lvert D_{\mathcal{R},\eps} u (x) \rvert^2.
\end{align*}
For the second inequality just take a $v \in \mathcal{A}_\eps(\Omega,0)$ and calculate
\[ \eps^d \sum_{x \in \into_\eps \Omega} u(x)v(x) \leq \lVert u \rVert_{\ell_\eps^2 (\into_\eps \Omega)} \lVert v \rVert_{\ell_\eps^2 (\into_\eps \Omega)} \leq C_{\rm P} \lVert u \rVert_{\ell_\eps^2 (\into_\eps \Omega)} \lVert v \rVert_{h^1_\eps (\sinto_\eps \Omega)}. \]
\end{proof}

Now let us prove the theorem.
\begin{proof}[Proof of Theorem \ref{thm:atomisticstatictheorem}]
By Theorem \ref{thm:stability}, $\lambda_{\rm atom}(A_0)>0$ implies $\lambda_{\rm LH}(A_0)>0$ and we can apply Theorem \ref{thm:contresult} with $m=2$. This already gives a $K_1$ and the solution of the continuous problem $y$. Since the solution depends continuously on the data and we have the embedding $H^4\hookrightarrow C^1$ we can always achieve
\[ \lvert D_{\mathcal{R},\eps} S_\eps y(x) - (A_0 \rho)_{\rho \in \mathcal{R}} \rvert \leq \frac{r_0}{2} \]
for all $x \in \eps \Z^d$ and all $\eps \in (0,1]$ by choosing $K_1$ small enough.

We want to use Corollary \ref{cor:quantitativeimplicitfunctiontheorem}. Let $X_\eps = \mathcal{A}_\eps(\Omega,0)$ with $\lVert u \rVert_{X_\eps} = \lVert u \rVert_{h^1_\eps(\sinto_\eps \Omega)}$,
\[Z_\eps = \{ r \colon \into_\eps \Omega \to \R^d \}\]
with $\lVert r \rVert_{Z_\eps} = \lVert r \rVert_{h^{-1}_\eps(\into_\eps \Omega)}$ and
\[Y_\eps =\Big( \big\{g \colon \partial_\eps \Omega \to \R^d\big\} \big/ \{g \colon \lVert g \rVert_{\partial_\eps \Omega} = 0\} \Big) \times \mathcal{A}_\eps(\Omega,0),\]
with $\lVert ([g],v) \rVert_{Y_\eps} = \lVert g \rVert_{\partial_\eps \Omega} + \lVert v \rVert_{h^{-1}_\eps(\into_\eps \Omega)}$. Note that $D_{\mathcal{R},\eps} T_\eps g (x) = D_{\mathcal{R},\eps} T_\eps h (x)$ for all $x \in \sinto_\eps \Omega$ whenever $[g]=[h]$. Now, define $F_\eps \colon X_\eps \times Y_\eps \to Z_\eps$ by
\begin{align*}
F_\eps (u,[g],v)(x) = - \tilde{f}(x) - v(x) -\divo_{\mathcal{R},\eps} \big(DW_{\rm atom}(D_{\mathcal{R},\eps} (S_\eps y + T_\eps g + u)(x))\big)
\end{align*}
for $x \in \into_\eps \Omega$. This is well defined for all $\eps \in (0,1]$ on an open neighborhood of
\[\overline{B_{r_1 \eps^{\frac{d}{2}}}(0)} \times \overline{B_{r_2 \eps^{\frac{d}{2}}}(0)} \times \mathcal{A}_\eps(\Omega,0),\]
if we choose $r_1,r_2>0$ small enough. In particular, we can choose them so small that
\[D_{\mathcal{R},\eps} (S_\eps y + T_\eps g + u)(x) \in \overline{B_{r_0}((A_0 \rho)_{\rho \in \mathcal{R}})}\]
for all $x \in \sinto_\eps \Omega$. Now we use Proposition \ref{prop:ell2residuum} with $S_\eps y \in C^{3,1}(\Omega;\R^d)$ and $V=B_{r_0}((A_0 \rho)_{\rho \in \mathcal{R}})$ to get
\begin{align*}
\lVert F_\eps(0,0,0) \rVert_{\ell^2_\eps(\into_\eps \Omega)} &\leq \lVert \divo DW_{\rm CB}(\nabla y) - \divo DW_{\rm CB}(\nabla S_\eps y) \rVert_{L^2(\Omega;\R^d)}\\
&+C\eps^2 \Big\lVert \lVert \nabla^4 S_\eps y \rVert_{L^\infty(B_{\eps R}(x))} + \lVert \nabla^3 S_\eps y \rVert_{L^\infty(B_{\eps R}(x))}^\frac{3}{2}\\
&+ \lVert \nabla^2 S_\eps y \rVert_{L^\infty(B_{\eps R}(x))}^3 + \eps \lVert \nabla^3 S_\eps y \rVert_{L^\infty(B_{\eps R}(x))}^2\Big\rVert_{L^2(\Omega)}.
\end{align*}
Next, we can apply Proposition \ref{prop:approximation1} and Proposition \ref{prop:approximation2} on $\bar{y} = y_{A_0} + E(y-y_{A_0})$ and use $y= \bar{y}$ in $\Omega$ to obtain
\begin{align*}
\lVert F_\eps(0,0,0) \rVert_{\ell^2_\eps(\into_\eps \Omega)} &\leq C \eps^2 \big( \lVert \nabla^2 \bar{y} \rVert_{L^4(\R^d)} \lVert \nabla^3 \bar{y} \rVert_{L^4(\R^d)} +  \lVert \nabla^4 \bar{y} \rVert_{L^2(\R^d)}\\
&+  \lVert \nabla^3 \bar{y} \rVert_{L^3(\R^d)}^\frac{3}{2} + \lVert \nabla^2 \bar{y} \rVert_{L^6(\R^d)}^3 + \eps \lVert \nabla^3 \bar{y} \rVert_{L^4(\R^d)}^2 \big)\\
&\leq C_1 \eps^2 \lVert y-y_{A_0} \rVert_{H^4(\Omega;\R^d)} (1+ \lVert y-y_{A_0} \rVert_{H^4(\Omega;\R^d)}^2).
\end{align*}
Hence, we can set
\[ A = C_{\rm P} C_1 \lVert y-y_{A_0} \rVert_{H^4(\Omega;\R^d)} (1+ \lVert y-y_{A_0} \rVert_{H^4(\Omega;\R^d)}^2).\]
By stability,
\[ \eps^d \sum\limits_{x \in \sinto_\eps \Omega} D^2 W_{\rm atom}((A_0 \rho)_{\rho \in \mathcal{R}}) [D_{\mathcal{R},\eps} u (x),D_{\mathcal{R},\eps} u (x)] \geq \lambda_{\rm atom}(A_0) \eps^d \sum\limits_{x \in \sinto_\eps \Omega} \lvert D_{\mathcal{R},\eps} u (x)\rvert^2\]
for all $u \in \mathcal{A}_\eps(\Omega,0)$. Continuity of $D^2 W_{\rm atom}$ then implies the existence of a $\tilde{r} \leq r_0$ such that
\[ \eps^d \sum\limits_{x \in \sinto_\eps \Omega} D^2 W_{\rm atom}(D_{\mathcal{R},\eps} w(x)) [D_{\mathcal{R},\eps} u (x),D_{\mathcal{R},\eps} u (x)] \geq \frac{\lambda_{\rm atom}(A_0)}{2} \eps^d \sum\limits_{x \in \sinto_\eps \Omega} \lvert D_{\mathcal{R},\eps} u (x)\rvert^2\]
for all $u \in \mathcal{A}_\eps(\Omega,0)$ and all $w \colon \Omega \cap \eps \Z^d$ with
\[\lvert D_{\mathcal{R},\eps} w(x) - (A_0 \rho)_{\rho \in \mathcal{R}} \rvert \leq \tilde{r}\]
for all $x \in \sinto_\eps \Omega$. And again, by choosing $K_1$ small enough this last inequality is automatically satisfied for $w= S_\eps y$ with $\eps \in (0,1]$ arbitrary.

Since the spaces are finite dimensional, it is obvious that the $F_\eps$ are Fréchet-differentiable. For $w \in X_\eps$ we have
\begin{align*}
\eps^d &\sum_{x \in \into_\eps\Omega} \partial_u F_\eps (0,0,0)[w](x) w(x)\\
 &= -\eps^{d-1} \sum_{x \in \into_\eps\Omega}\sum_{\sigma, \rho \in \mathcal{R}} w(x) \Big( D_{e_\rho}D_{e_\sigma}W_{\rm atom}(D_{\mathcal{R},\eps}S_\eps y(x)) \frac{w(x+\eps \sigma) - w(x)}{\eps}\\
 &\quad - D_{e_\rho}D_{e_\sigma}W_{\rm atom}(D_{\mathcal{R},\eps}S_\eps y(x- \eps \rho)) \frac{w(x+\eps \sigma- \eps \rho) - w(x- \eps \rho)}{\eps}\Big)\\
 &=\eps^d \sum\limits_{x \in \sinto_\eps \Omega} D^2 W_{\rm atom}(D_{\mathcal{R},\eps} S_\eps y(x)) [D_{\mathcal{R},\eps} w (x),D_{\mathcal{R},\eps} w (x)] \\
&\geq \frac{\lambda_{\rm atom}(A_0)}{2} \eps^d \sum\limits_{x \in \sinto_\eps \Omega} \lvert D_{\mathcal{R},\eps} w(x)\rvert^2
\end{align*}
and, thus,
\[ \lVert \partial_u F_\eps (0,0,0)^{-1} \rVert_{L(h^{-1}_\eps(\into_\eps \Omega), h_\eps^1(\sinto_\eps \Omega))} \leq \frac{2}{\lambda_{\rm atom}(A_0)} = M_1.\]
Furthermore, for $([h],v) \in Y_\eps$ and $w \in \mathcal{A}_\eps(\Omega,0)$ we have
\begin{align*}
\eps^d &\sum_{x \in \into_\eps \Omega} \partial_{([g],v)}F_\eps(0,0,0)[([h],v)](x)w(x)\\
&\leq \lVert v \rVert_{h^{-1}_\eps(\into_\eps\Omega)} \lVert w \rVert_{h^1_\eps(\sinto_\eps\Omega)} \\
&\quad - \eps^{d-1} \sum_{x \in \into_\eps \Omega} \sum_{\sigma, \rho \in \mathcal{R}} w(x) \Big( D_{e_\rho}D_{e_\sigma}W_{\rm atom} (D_{\mathcal{R},\eps} S_\eps y(x)) \frac{T_\eps h(x+\eps \sigma) - T_\eps h(x)}{\eps}\\
&\quad- D_{e_\rho}D_{e_\sigma}W_{\rm atom} (D_{\mathcal{R},\eps} S_\eps y (x-\eps \rho)) \frac{T_\eps h(x-\eps \rho+\eps \sigma) - T_\eps h(x - \eps \rho)}{\eps} \Big) \\
&= \lVert v \rVert_{h^{-1}_\eps(\into_\eps\Omega)} \lVert w \rVert_{h^1_\eps(\sinto_\eps\Omega)}\\
&\quad + \eps^d \sum_{x \in \sinto_\eps \Omega} D^2W_{\rm atom}(D_{\mathcal{R},\eps} S_\eps y(x))[D_{\mathcal{R},\eps} w(x),D_{\mathcal{R},\eps} T_\eps h(x)]\\
&\leq \lVert v \rVert_{h^{-1}_\eps(\into_\eps\Omega)} \lVert w \rVert_{h^1_\eps(\sinto_\eps\Omega)} + \lVert D^2 W_{\rm atom} \rVert_\infty \lVert w \rVert_{h^1_\eps(\sinto_\eps\Omega)} \lVert h \rVert_{\partial_\eps \Omega}.
\end{align*}
Hence,
\[ \lVert \partial_{([g],v)}F_\eps(0,0,0) \rVert_{L(Y_\eps,Z_\eps)} \leq 1 + \lVert D^2 W_{\rm atom} \rVert_\infty = M_2.\]
In a similar fashion we calculate
\begin{align*}
\eps^d &\sum_{x \in \into_\eps \Omega} \big(\partial_u F_\eps(0,0,0) - \partial_u F_\eps(u,[g],v)\big)[w](x)z(x)\\
&= \eps^d \sum_{x \in \sinto_\eps \Omega} \big(D^2 W_{\rm atom}(D_{\mathcal{R},\eps} S_\eps y (x))\\
&\qquad -D^2 W_{\rm atom}(D_{\mathcal{R},\eps} (S_\eps y +u+T_\eps g) (x))  \big)[D_{\mathcal{R},\eps} w(x), D_{\mathcal{R},\eps} z(x)]\\
&\leq  \lVert w \rVert_{h^1_\eps(\sinto_\eps\Omega)} \lVert z \rVert_{h^1_\eps(\sinto_\eps\Omega)} \lVert D^3 W_{\rm atom} \rVert_\infty \lVert D_{\mathcal{R},\eps} (u+T_\eps g) \rVert_{\ell^\infty (\sinto_\eps \Omega)}.
\end{align*}
Thus,
\[ \lVert \partial_u F_\eps(0,0,0) - \partial_u F_\eps(u,[g],v) \rVert_{L(X_\eps, Z_\eps)} \leq \lVert D^3 W_{\rm atom} \rVert_\infty \eps^{-\frac{d}{2}} (\lVert u \rVert_{h^1_\eps(\sinto_\eps\Omega)} + \lVert g \rVert_{\partial_\eps \Omega}), \]
so that we can take $M_3 = \lVert D^3 W_{\rm atom} \rVert_\infty$.
At last,
\begin{align*}
\eps^d &\sum_{x \in \into_\eps \Omega} \big(\partial_{([g],v)} F_\eps(0,0,0) - \partial_{([g],v)} F_\eps(u,[g],v)\big)[([h],w)](x)z(x)\\
&= \eps^d \sum_{x \in \sinto_\eps \Omega} \big(D^2 W_{\rm atom}(D_{\mathcal{R},\eps} S_\eps y(x))\\
&\qquad -D^2 W_{\rm atom}(D_{\mathcal{R},\eps} (S_\eps y +u+T_\eps g) (x))  \big)[D_{\mathcal{R},\eps} T_\eps h(x), D_{\mathcal{R},\eps} z(x)]\\
&\leq  \lVert h \rVert_{\partial_\eps \Omega} \lVert z \rVert_{h^1_\eps(\sinto_\eps\Omega)} \lVert D^3 W_{\rm atom} \rVert_\infty \lVert D_{\mathcal{R},\eps} (u+T_\eps g) \rVert_{\ell^\infty (\sinto_\eps \Omega)}.
\end{align*}
Hence, we can also take $M_4 = \lVert D^3 W_{\rm atom} \rVert_\infty$. As before, since $y$ depends continuously on the data, we can take $K_1$ small enough such that
\[ C_{\rm P} C_1 \lVert y-y_{A_0} \rVert_{H^4(\Omega;\R^d)} (1+ \lVert y-y_{A_0} \rVert_{H^4(\Omega;\R^d)}^2) \leq \min\Big\{ \frac{r_1 \lambda_{\rm atom}(A_0)}{8}, \frac{\lambda_{\rm atom}(A_0)^2}{64 \lVert D^3 W_{\rm atom} \rVert_\infty}\Big\}.\]
Therefore, we can apply Corollary \ref{cor:quantitativeimplicitfunctiontheorem} and get the fixed point result with
\begin{align*}
	\lambda_1 &= \min\Big\{r_1,\frac{\lambda_{\rm atom}(A_0)}{8 \lVert D^3 W_{\rm atom} \rVert_\infty}\Big\},\\
	\lambda_2 &= \min\Big\{r_2,\frac{\lambda_{\rm atom}(A_0)}{8\lVert D^3 W_{\rm atom} \rVert_\infty}, \frac{r_1 \lambda_{\rm atom}(A_0)}{8(1 + \lVert D^2 W_{\rm atom} \rVert_\infty) + 2 \lambda_{\rm atom}(A_0)}\\
	&\quad \frac{\lambda_{\rm atom}(A_0)^2}{8 \lVert D^3 W_{\rm atom} \rVert_\infty \big(8 + 8\lVert D^2 W_{\rm atom} \rVert_\infty + 2\lambda_{\rm atom}(A_0)\big)}\Big\}.
\end{align*}
After doing the substitutions $g_{\rm atom} \in S_\eps y + [g]$, $f_{\rm atom} = \tilde{f} +v$ and $y_{\rm atom} = S_\eps y + T_\eps (g_{\rm atom}-S_\eps y) + u$, we get the stated existence result with $K_2=\frac{\lambda_2}{2}$. The solution then satisfies $\lVert y_{\rm atom} - S_{\eps}y \rVert_{h^1_\eps(\sinto_\eps \Omega)} \leq K_3 \eps^\gamma$ with $K_3 = \lambda_1 + \frac{\lambda_2}{2}$.
If $r_1,r_2$ are chosen small enough, then $\lVert \tilde{y}_{\rm atom} - S_{\eps}y \rVert_{h^1_\eps(\sinto_\eps \Omega)} \leq K_3 \eps^\gamma$ implies
\[ \lvert D_{\mathcal{R},\eps} \tilde{y}_{\rm atom}(x) - (A_0 \rho)_{\rho \in \mathcal{R}} \rvert \leq \frac{\tilde{r}}{2} \]
for all $x \in \sinto_\eps \Omega$ and any $\tilde{y}_{\rm atom}$ with boundary values $g_{\rm atom}$. Furthermore, with $r_1,r_2$ chosen small enough for $u \in \mathcal{A}_\eps(\Omega, 0)\backslash\{0\}$ with $\lVert u \rVert_{h^1_\eps(\sinto_\eps \Omega)} \leq K_3 \eps^\gamma$ we have
\[ \lvert D_{\mathcal{R},\eps} u(x) \rvert \leq \frac{\tilde{r}}{2}. \]
Now, since $y_{\rm atom}$ is a solution, we can calculate
\begin{align*}
&E_\eps (y_{\rm atom} + u, f_{\rm atom}, g_{\rm atom})-E_\eps(y_{\rm atom}, f_{\rm atom}, g_{\rm atom})\\
&= \eps^d \sum_{x \in \sinto_\eps \Omega} \Big( W_{\rm atom}(D_{\mathcal{R},\eps} y_{\rm atom}(x) + D_{\mathcal{R},\eps} u(x)) - W_{\rm atom}(D_{\mathcal{R},\eps} y_{\rm atom}(x))\\
&\qquad - DW_{\rm atom}(D_{\mathcal{R},\eps} y_{\rm atom}(x))[D_{\mathcal{R},\eps} u(x)] \Big) \\
&= \eps^d \sum_{x \in \sinto_\eps \Omega}  \int_0^1 (1-t) D^2 W_{\rm atom}(D_{\mathcal{R},\eps} y_{\rm atom}(x) + t D_{\mathcal{R},\eps} u(x))[D_{\mathcal{R},\eps} u(x),D_{\mathcal{R},\eps} u(x)] \,dt\\
&\geq \frac{\lambda_{\rm atom}(A_0)}{2} \eps^d \sum\limits_{x \in \sinto_\eps \Omega} \lvert D_{\mathcal{R},\eps} u (x)\rvert^2 > 0,
\end{align*}
which shows that $y_{\rm atom}$ is a strict local minimizer. And, doing the same calculation again with $\tilde{y}_{\rm atom} - y_{\rm atom}$ instead of $u$, we also see that the solution is unique.

For the additional statement we only have to estimate $\lVert S_\eps y - y\rVert_{h_\eps^1(\into_\eps \Omega)}$ with a Taylor expansion. We have
\begin{align*}
S_\eps y(x+\eps \rho)&- y(x+\eps \rho) = \int_{B_\eps(0)} \big(\bar{y}(x+\eps \rho +z)- \bar{y}(x+\eps \rho) \Big)\eta_\eps(z)\,dz\\
&= \int_{B_\eps(0)} \big(\bar{y}(x+\eps \rho +z)- \bar{y}(x+\eps \rho) - \nabla \bar{y}(x+\eps \rho)[z] \big)\eta_\eps(z)\,dz\\
&= \int_{B_\eps(0)}\int_0^1 (1-t) \nabla^2 \bar{y}(x+\eps \rho +tz)[z,z] \eta_\eps(z)\,dt\,dz.
\end{align*}
This includes the case $\rho =0$, hence
\begin{align*}
\Big\lvert &\frac{S_\eps y(x+\eps \rho) - S_\eps y(x)}{\eps} - \frac{y(x+\eps \rho) -y(x)}{\eps}\Big\rvert^2 \\
&= \Big\lvert \int_{B_\eps(0)}\int_0^1(1-t) \frac{\nabla^2 \bar{y}(x+\eps \rho+tz)[z,z]  - \nabla^2 \bar{y}(x+tz)[z,z]}{\eps}\eta_\eps(z)\,dt \,dz \Big\rvert^2\\
&\leq \eps^{2\gamma} R_{\rm max}^{2(\gamma-1)} \lvert \nabla^2 \bar{y} \rvert_{\gamma-1}^2,
\end{align*}
which gives the desired result.
\end{proof}

\bibliographystyle{myalpha}
\bibliography{papers}

\end{document}